\documentclass{amsart}

\usepackage{amssymb}
\usepackage[table]{xcolor}
\usepackage{marvosym}
\usepackage{tikz}

\usepackage{array}
\newcolumntype{P}[1]{>{\centering\arraybackslash}p{#1}}

\theoremstyle{plain}
\newtheorem{theorem}{Theorem}[section]

\newtheorem{corollary}[theorem]{Corollary}

\newtheorem{lemma}[theorem]{Lemma}

\newtheorem{proposition}[theorem]{Proposition}

\newtheorem{problem}[theorem]{Problem}

\theoremstyle{definition}

\newtheorem{example}[theorem]{Example}
\newtheorem{remark}[theorem]{Remark}

\newtheorem{definition}[theorem]{Definition}

\numberwithin{equation}{section}

\def\ldiv{\backslash}
\def\rdiv{/}
\def\sym#1{\mathrm{Sym}(#1)}

\def\dldiv{\backslash\!\!\backslash}
\def\drdiv{/\!\!/}

\def\F{\mathbb F}
\def\D{\mathbb D}

\def\DS{\mathrm{DS}}

\newcommand{\cref}[1]{Corollary~\ref{#1}}

\def\cc{\cellcolor{gray!25}} 

\def\settablename#1{
    \def\tblnm{#1}
}

\def\namedsudoku#1#2#3#4#5#6#7#8#9
{
\begingroup
\addtolength{\tabcolsep}{-3pt}
\renewcommand{\arraystretch}{0.975}
\begin{scriptsize}
\begin{tabular}{|@{\hspace{1.2em}}ccc@{\hspace{1.4em}}ccc@{\hspace{1.4em}}ccc@{\hspace{1.2em}}|}
\multicolumn{9}{c}{\normalsize{\tblnm}}\\ [1ex]
\hline
&&&&&&&&\\
#1\\
#2\\
#3\\ 
&&&&&&&&\\
#4\\
#5\\
#6\\ 
&&&&&&&&\\
#7\\
#8\\
#9\\ 
&&&&&&&&\\
\hline
\end{tabular}
\end{scriptsize}
\endgroup
}

\def\sudoku#1#2#3#4#5#6#7#8#9
{
\begingroup
\addtolength{\tabcolsep}{-3pt}
\renewcommand{\arraystretch}{0.975}
\begin{scriptsize}
\begin{tabular}{|@{\hspace{1.2em}}ccc@{\hspace{1.4em}}ccc@{\hspace{1.4em}}ccc@{\hspace{1.2em}}|}
\hline
&&&&&&&&\\
#1\\
#2\\
#3\\ 
&&&&&&&&\\
#4\\
#5\\
#6\\ 
&&&&&&&&\\
#7\\
#8\\
#9\\ 
&&&&&&&&\\
\hline
\end{tabular}
\end{scriptsize}
\endgroup
}

\def\namedsmallsudoku#1#2#3#4#5#6#7#8#9
{
\begingroup
\addtolength{\tabcolsep}{-4.0pt}
\renewcommand{\arraystretch}{0.82}
\begin{tiny}
\begin{tabular}{|@{\hspace{0.8em}}ccc@{\hspace{1.2em}}ccc@{\hspace{1.2em}}ccc@{\hspace{0.8em}}|}
\multicolumn{9}{c}{\tblnm}\\ [0.6ex]
\hline
&&&&&&&&\\
#1\\
#2\\
#3\\
&&&&&&&&\\
#4\\
#5\\
#6\\
&&&&&&&&\\
#7\\
#8\\
#9\\
&&&&&&&&\\
\hline
\end{tabular}
\end{tiny}
\endgroup
}

\def\smallsudoku#1#2#3#4#5#6#7#8#9
{
\begingroup
\addtolength{\tabcolsep}{-4.0pt}
\renewcommand{\arraystretch}{0.82}
\begin{tiny}
\begin{tabular}{|@{\hspace{0.8em}}ccc@{\hspace{1.2em}}ccc@{\hspace{1.2em}}ccc@{\hspace{0.8em}}|}
\hline
&&&&&&&&\\
#1\\
#2\\
#3\\
&&&&&&&&\\
#4\\
#5\\
#6\\
&&&&&&&&\\
#7\\
#8\\
#9\\
&&&&&&&&\\
\hline
\end{tabular}
\end{tiny}
\endgroup
}

\def\oplabel#1{
    \def\opl{#1}
}

\def\rowlabels#1#2#3#4#5#6#7#8#9{
    \def\rla{#1}
    \def\rlb{#2}
    \def\rlc{#3}
    \def\rld{#4}
    \def\rle{#5}
    \def\rlf{#6}
    \def\rlg{#7}
    \def\rlh{#8}
    \def\rli{#9}
}

\def\collabels#1#2#3#4#5#6#7#8#9{
    \def\cla{#1}
    \def\clb{#2}
    \def\clc{#3}
    \def\cld{#4}
    \def\cle{#5}
    \def\clf{#6}
    \def\clg{#7}
    \def\clh{#8}
    \def\cli{#9}
}

\def\smalllabeledsudoku#1#2#3#4#5#6#7#8#9 
{
\begingroup
\addtolength{\tabcolsep}{-4.0pt}
\renewcommand{\arraystretch}{0.82}
\begin{tiny}
\begin{tabular}{c@{\hspace{0.4em}}|p{0.4em}@{\hspace{0.4em}}ccc@{\hspace{1.2em}}ccc@{\hspace{1.2em}}ccc@{\hspace{0.8em}}|}  
\multicolumn{1}{c}{\opl}&&\cla&\clb&\clc&\cld&\cle&\clf&\clg&\clh&\multicolumn{1}{c@{\hspace{0.8em}}}{\cli}\\[0.6ex]
\cline{2-11}
&&&&&&&&&&\\
\rla&&#1\\
\rlb&&#2\\
\rlc&&#3\\
&&&&&&&&&&\\
\rld&&#4\\
\rle&&#5\\
\rlf&&#6\\
&&&&&&&&&&\\
\rlg&&#7\\
\rlh&&#8\\
\rli&&#9\\
&&&&&&&&&&\\
\cline{2-11}
\end{tabular}
\end{tiny}
\endgroup
}

\begin{document}

\title[Division Sudokus]{Division Sudokus: Invariants, Enumeration and Multiple Partitions}

\author{Ale\v s Dr\'apal}
\address[Dr\'apal]{Department of Mathematics, Charles University, Sokolovsk\'a 83, 186 75, Praha 8, Czech Republic}
\email{drapal@karlin.mff.cuni.cz}

\author{Petr Vojt\v echovsk\' y}
\address[Vojt\v{e}chovsk\'y]{Department of Mathematics, University of Denver, 2390 S.~York St, Denver, Colorado, 80208, USA}
\email{petr@math.du.edu}

\begin{abstract}
A division sudoku is a latin square whose all six conjugates are sudoku squares. We enumerate division sudokus up to a suitable equivalence, introduce powerful invariants of division sudokus, and also study latin squares that are division sudokus with respect to multiple partitions at the same time. We use nearfields and affine geometry to construct division sudokus of prime power rank that are rich in sudoku partitions.
\end{abstract}



\thanks{Petr Vojt\v{e}chovsk\'y partially supported by the 2019 PROF grant of the University of Denver.}

\maketitle

\section{Introduction}

Sudoku is the most popular puzzle of our time. Many nontrivial results about sudoku squares and sudoku puzzles have been obtained by mathematicians and sudoku enthusiasts \cite{BaileyEtAl, CameronEtAl, Keedwell, Lorch, McGuireEtAl, PedersenVis, SuDokusMaths}.

In this paper we study a special type of sudokus that we propose to call \emph{division sudokus}. Division sudokus can be thought of as sudokus for which the roles of rows, columns and symbols are fully interchangeable. We have chosen the name division sudoku because a latin square is a division sudoku if and only if, when viewed as a quasigroup, the operation tables for multiplication, left division and right division are all sudokus. (See Definition \ref{Df:DivisionSudoku} for a more precise statement and Proposition \ref{Pr:DSChar} for a list of equivalent conditions.)

It appears that division sudokus were for the first time considered by G\"unter Stertenbrink in 2005 \cite[p.\,27]{SuDokusMaths} under the name 3dokus, which has since come to mean something else, namely a generalization of sudokus to three dimensions \cite{3Doku, LambertWhitlock}. Stertenbrink correctly determined the number of standard division sudokus and asked several natural questions concerning enumeration and completability of division sudokus. Bailey, Cameron and Connelly \cite{BaileyEtAl} looked at a small class of sudokus (so called symmetric sudoku puzzles) from a design-theoretic point of view, and they discussed an example \cite[Figure 5]{BaileyEtAl} that happens to be a division sudoku.

We arrived at the notion of division sudokus independently by studying properties of the following, very interesting example of a sudoku square:

\medskip
\begin{center}
\settablename{$L_0$}
\namedsudoku{1&4&7&3&8&5&2&9&6}{8&2&5&6&1&9&4&3&7}{6&9&3&7&4&2&8&5&1}{5&3&9&4&7&1&6&2&8}{7&6&1&2&5&8&9&4&3}{2&8&4&9&3&6&1&7&5}{9&5&2&8&6&3&7&1&4}{3&7&6&1&9&4&5&8&2}{4&1&8&5&2&7&3&6&9}
\end{center}
\medskip

Not only is $L_0$ a division sudoku, but it happens to be the ``least associative'' latin square of order $9$. In more detail, an \emph{associative triple} in a quasigroup $(X,\cdot)$ is an ordered triple $(x,y,z)$ such that $x\cdot(y\cdot z) = (x\cdot y)\cdot z$. It is easy to show that in a quasigroup of order $n$ the number of associative triples is at least $n$. Until recently, it was an open problem to find a quasigroup of order $n>1$ containing precisely $n$ associative triples \cite{Kepka}. No such quasigroup exists for $n\le 8$ \cite{DVJCD}. With rows and columns labeled by $1$, $\dots$, $9$ in this order, the square $L_0$ is a multiplication table of a unique quasigroup of order $9$ up to isomorphism with precisely $9$ associative triples \cite{DL,DV}.

\bigskip

A \emph{sudoku of rank} $m$ is a latin square of order $m^2$ consisting of $m^2$ subsquares, each of size $m\times m$, such that each subsquare contains all $m^2$ symbols. The usual sudoku is therefore a sudoku of rank $3$. Not much is known about sudokus of rank $m>3$. In contrast, division sudokus are sufficiently restrictive so that many of their aspects can be investigated for larger ranks. Most of this paper is devoted to division sudokus of rank $3$ but it may well turn out that questions about division sudokus of larger rank will lead to more interesting mathematics. We start this line of investigation in Section \ref{Sc:Nearfields}.

The content of this paper is as follows. In Section \ref{Sc:Char} we characterize division sudokus in various ways. In Section \ref{Sc:Invariants} we introduce two powerful invariants of division sudokus of rank $3$, namely the intercalate structure invariant and the minisquare structure invariant. We use these invariants and several computational tools in Section \ref{Sc:Enumeration} to enumerate division sudokus of rank $3$ up to suitable equivalences. There are $104015259648$ standard division sudokus of rank $3$ (confirming Stertenbrink's result) forming 186 ds-isotopism classes (see Definition \ref{Df:DsIsotopic}), $45$ main ds-classes, and belonging to $183$ isotopism classes of quasigroups.

In Section \ref{Sc:MultiplePartitions} we for the first time consider separate sudoku partitions for rows, columns and symbols, so-called sudoku tri-partitions. (This paper could have started by defining division sudokus via tri-partitions. We opted for quasigroups since they provide a more formal setting for concepts like isotopisms.) We find a division sudoku of rank $3$ with $24$ sudoku tri-partitions. We then show how to synchronize as many sudoku tri-partitions as possible by group actions to obtain division sudokus with the largest number of sudoku partitions. There are only $3$ essentially different division sudokus of rank $3$ with more than one sudoku partition, namely with $3$, $4$ and $4$ sudoku partitions, respectively.

Building upon ideas of Stein \cite{Stein}, in Section \ref{Sc:Nearfields} we show how to construct division sudokus of prime power rank $m$ that have at least $m+1$ sudoku partitions. The construction is based on nearfields, affine planes and subspaces. The underlying set is the finite field of order $m^2$ or, when $m$ is odd, the quadratic nearfield of order $m^2$. If the rank $m$ is a prime power of the form $p^{2s}$, we can improve the lower bound on the number of sudoku partitions to $m^2+m$. For instance, there exists a division sudoku of rank $4$ with $20$ sudoku partitions.

Precisely one ds-isotopism class of division sudokus of rank $3$ consists of sudokus isotopic to a group (namely to $C_3\times C_3$). This class contains the example in \cite[Figure 5]{BaileyEtAl}.

\subsection{Notation and terminology}

We start by recalling the standard terminology for latin squares and quasigroups. We refer the reader to \cite{Bruck, Handbook} for more details.

Let $X$ be a finite set. A \emph{latin square} on $X$ is an array $L=(L(x,y))_{x,y}$ with entries $L(x,y)\in X$ such that every row of $L$ contains all symbols of $X$ and every column of $L$ contains all symbols of $X$. The algebraic counterpart of a latin square is a quasigroup. A \emph{quasigroup} on $X$ is a pair $(X,\cdot)$, where $\cdot$ is a binary operation on $X$ such that for every $a$, $b\in X$ there are unique $x$, $y\in X$ satisfying $a\cdot x=b$ and $y\cdot a=b$. Once a labeling of rows and columns is fixed, there is a one-to-one correspondence between latin squares on $X$ and (multiplication tables of) quasigroups on $X$. Throughout the paper, when $L$ is a latin square on $X=\{1,2,\dots,n\}$, the associated quasigroup is obtained from $L$ by labeling the rows and columns by $1$, $\dots$, $n$, in this order.

A quasigroup can be defined equationally as a set $X$ with three binary operations $\cdot$, $\rdiv$, $\ldiv$ such that the identities $(x\cdot y)\rdiv y=x$, $(x\rdiv y)\cdot y=x$, $x\cdot(x\ldiv y)=y$ and $x\ldiv (x\cdot y)=y$ are satisfied. We refer to $\rdiv$ and $\ldiv$ as \emph{right division} and \emph{left division}, respectively. If $(Q,\cdot)$ is a quasigroup, so are $(Q,\rdiv)$ and $(Q,\ldiv)$.

The \emph{orthogonal array} of a latin square $L$ on $X$ is the set
\begin{displaymath}
    O(L)=\{(x,y,L(x,y)):x,\,y\in X\}\subseteq X\times X\times X.
\end{displaymath}
The latin square $L$ can be reconstructed from $O(L)$ by setting $L(x,y)$ to be the unique $z$ such that $(x,y,z)\in O(L)$.

\medskip

The symmetric group on a set $X$ will be denoted by $\sym{X}$ and we also write $\sym{\{1,\dots,n\}} = \sym{n}$. Given $\theta\in\sym{3}$ and a latin square $L$, we denote by $L^\theta$ the latin square whose orthogonal array is equal to
\begin{displaymath}
    \{(x_{\theta(1)},x_{\theta(2)},x_{\theta(3)}):(x_1,x_2,x_3)\in O(L)\}.
\end{displaymath}
For instance, $L^{(1,2)}$ is the matrix transpose of $L$. The six not necessarily distinct latin squares $L^\theta$ for $\theta\in\sym{3}$ are called the \emph{conjugates} of $L$. Algebraically speaking, the conjugates of $L=(X,\cdot)$ are precisely the quasigroups
\begin{displaymath}
    L=(X,\cdot),\quad L^{(1,3)}=(X,\rdiv),\quad L^{(2,3)}=(X,\ldiv)
\end{displaymath}
and their opposites. Indeed, we have $x_1\rdiv x_2=x_3$ iff $x_3x_2=x_1$ iff $(x_3,x_2,x_1)\in O(L)$, so $L^{(1,3)} = (X,\rdiv)$, and similarly for the left division.

An \emph{isotopism} from a latin square $L_1$ to a latin square $L_2$ on $X$ is an ordered triple $(\alpha,\beta,\gamma)$ of permutations of $X$ such that
\begin{displaymath}
    L_2(\alpha(x),\beta(y)) = \gamma( L_1(x,y) )
\end{displaymath}
holds for every $x$, $y\in X$, in which case the latin squares $L_1$, $L_2$ are said to be \emph{isotopic}. Hence two latin squares are isotopic if and only if one is obtained from the other by a permutation of rows, columns and symbols. Since isotopisms preserve combinatorial properties of latin squares, the classification of latin squares is often pursued up to isotopism. From the point of view of orthogonal arrays, it is also natural to permute the roles of rows, columns and symbols. Two latin squares are said to be \emph{paratopic} if one is isotopic to a conjugate of the other. The equivalence classes of latin squares induced by paratopy are called \emph{main classes}.

\medskip

Let us now introduce notation and terminology for sudokus. We will rely in part on the wonderfully descriptive concepts used by sudoku enthusiasts.

Let $X$ be a set of size $m^2$. A partition $\{X_1,\dots,X_m\}$ of $X$ is a \emph{sudoku partition} if $|X_i|=m$ for every $1\le i\le m$. The subsets $X_i$ are called \emph{blocks}. A sudoku partition is \emph{standard} if
\begin{displaymath}
    X_i = \{(i-1)m+1,(i-1)m+2,\dots,im\}
\end{displaymath}
for every $1\le i\le m$. A permutation $\theta\in\sym{X}$ is said to \emph{preserve the partition} $\{X_1,\dots,X_m\}$ if the image $\theta(X_i)$ of every block $X_i$ is a block.

When $m=3$, we take advantage of the fact that $m^2<10$ and use compact notation for sudoku partitions. For instance, $\{123\ 456\ 789\}$ will denote the sudoku partition $\{\{1,2,3\},\{4,5,6\},\{7,8,9\}\}$.

Let $L$ be a latin square on $X$ and let $O(L)$ be its orthogonal array. The \emph{band} of $L$ corresponding to the block $X_i$ is the set
\begin{displaymath}
    B_i = \{(x,y,z)\in O(L):x\in X_i\},
\end{displaymath}
the \emph{stack} of $L$ corresponding to $X_i$ is the set
\begin{displaymath}
    S_i = \{(x,y,z)\in O(L):y\in X_i\},
\end{displaymath}
and the \emph{pile} of $L$ corresponding to $X_i$ is the set
\begin{displaymath}
    P_i=\{(x,y,z)\in O(L):z\in X_i\}.
\end{displaymath}
An intersection of a band and a stack is a \emph{minisquare}. The rows and columns in a minisquare are \emph{minirows} and \emph{minicolumns}, respectively. In the following figure we have highlighted the band $B_2$, stack $S_2$ and pile $P_2$ in the standard division sudoku $L_0$.

\medskip
\begin{center}
\begin{tabular}{P{30mm}P{30mm}P{30mm}}
\smallsudoku{1&4&7&3&8&5&2&9&6}{8&2&5&6&1&9&4&3&7}{6&9&3&7&4&2&8&5&1}{\cc5&\cc3&\cc9&\cc4&\cc7&\cc1&\cc6&\cc2&\cc8}{\cc7&\cc6&\cc1&\cc2&\cc5&\cc8&\cc9&\cc4&\cc3}{\cc2&\cc8&\cc4&\cc9&\cc3&\cc6&\cc1&\cc7&\cc5}{9&5&2&8&6&3&7&1&4}{3&7&6&1&9&4&5&8&2}{4&1&8&5&2&7&3&6&9}
&\smallsudoku{1&4&7&\cc3&\cc8&\cc5&2&9&6}{8&2&5&\cc6&\cc1&\cc9&4&3&7}{6&9&3&\cc7&\cc4&\cc2&8&5&1}{5&3&9&\cc4&\cc7&\cc1&6&2&8}{7&6&1&\cc2&\cc5&\cc8&9&4&3}{2&8&4&\cc9&\cc3&\cc6&1&7&5}{9&5&2&\cc8&\cc6&\cc3&7&1&4}{3&7&6&\cc1&\cc9&\cc4&5&8&2}{4&1&8&\cc5&\cc2&\cc7&3&6&9}
&\smallsudoku{1&\cc4&7&3&8&\cc5&2&9&\cc6}{8&2&\cc5&\cc6&1&9&\cc4&3&7}{\cc6&9&3&7&\cc4&2&8&\cc5&1}{\cc5&3&9&\cc4&7&1&\cc6&2&8}{7&\cc6&1&2&\cc5&8&9&\cc4&3}{2&8&\cc4&9&3&\cc6&1&7&\cc5}{9&\cc5&2&8&\cc6&3&7&1&\cc4}{3&7&\cc6&1&9&\cc4&\cc5&8&2}{\cc4&1&8&\cc5&2&7&3&\cc6&9}
\end{tabular}
\end{center}
\medskip

We will use informal terminology for subsets of orthogonal arrays. For instance, we say that $S\subseteq O(L)$ contains all symbols of $X$ if for every $z\in X$ there exist $x$, $y\in X$ such that $(x,y,z)\in S$.

We use the icon \Keyboard\ to indicate results obtained by computer computation. The calculations were carried out with \texttt{Mace4} \cite{Mace4} and \texttt{GAP} \cite{GAP}, some using the package \texttt{LOOPS} \cite{LOOPS}.

\subsection{Sudokus}

\begin{definition}
Let $\{X_1,\dots,X_m\}$ be a sudoku partition of $X$. A \emph{sudoku (square) of rank $m$} is a latin square $L$ on $X$ such that every minisquare of $L$ contains all symbols of $X$. A sudoku is \emph{standard} if the partition is standard.
\end{definition}

The following result goes back to $2003$ and it was confirmed independently several times.

\begin{theorem}[Stertenbrink \Keyboard]
There are $6670903752021072936960$ standard sudokus of rank $3$.
\end{theorem}

Just as the notions of isotopy and paratopy are the correct concepts for classifying latin squares and orthogonal arrays while preserving their combinatorial properties, so are the following two notions the correct concepts for classifying sudokus:

\begin{definition}
Let $\{X_1,\dots,X_m\}$ be a sudoku partition of $X$. Two sudokus $L_1$, $L_2$ defined on $X$ are \emph{s-isotopic} (short for \emph{sudoku-isotopic}) if there exists an isotopism $(\alpha,\beta,\gamma)$ from $L_1$ to $L_2$ such that $\alpha$, $\beta$ are partition preserving. Two sudokus $L_1$, $L_2$ are \emph{s-paratopic} (short for \emph{sudoku-paratopic}) if $L_1$ is s-isotopic to $L_2$ or to $L_2^{(1,2)}$.
\end{definition}

It is natural in the context of sudokus to demand that the permutations $\alpha$, $\beta$ be partition preserving, since the definition of sudoku depends on a fixed partition of rows and columns. The symbols of a sudoku are not partitioned in any way, however, hence the permutation $\gamma$ is not restricted. Similarly, it is natural to consider only the conjugates $L$ and $L^{(1,2)}$ of a sudoku $L$, not all six conjugates as in the case of latin square paratopisms, because the roles of rows and columns are symmetric in the definition of sudoku while the role of symbols is different.

\begin{theorem}[Russell and Jarvis {{\cite{RussellJarvis}}} \Keyboard]
There are $10945437157$ standard sudokus of rank $3$ up to s-isotopism and $5472730538$ standard sudokus of rank $3$ up to s-paratopism.
\end{theorem}

\section{Division sudokus and their characterization}\label{Sc:Char}

We are interested in those sudokus for which rows, columns and symbols play symmetric roles.

\begin{definition}\label{Df:DivisionSudoku}
Let $\{X_1,\dots,X_m\}$ be a sudoku partition of $X$. A \emph{division sudoku (square) of rank $m$} is a latin square $L$ on $X$ such that the intersection of every band with every stack contains all symbols of $X$, the intersection of every band with every pile contains all columns of $X$, and the intersection of every stack with every pile contains all rows of $X$.
\end{definition}

\begin{proposition}\label{Pr:DSChar}
Let $\{X_1,\dots,X_m\}$ be a sudoku partition of $X$. Then the following conditions are equivalent for a latin square $L$ on $X$:
\begin{enumerate}
\item[(i)] $L$ is a division sudoku,
\item[(ii)] the conjugates $L$, $L^{(1,3)}$ and $L^{(2,3)}$ are sudokus,
\item[(iii)] all conjugates of $L$ are sudokus,
\item[(iv)] $L$ is a sudoku in which every minirow and every minicolumn contains precisely one entry from every pile.
\end{enumerate}
\end{proposition}
\begin{proof}
The definition of division sudoku says that $L$, $L^{(1,3)}$ and $L^{(2,3)}$ are sudokus. A latin square $L$ is a sudoku if and only if the conjugate $L^{(1,2)}$ is a sudoku. Hence (i), (ii) and (iii) are equivalent.

Let now $L$ be a sudoku of rank $m$ and fix a band $B_i$ and a pile $P_j$. Every minisquare contained in $B_i$ contains all symbols, hence precisely $m$ symbols from the pile $P_j$. Altogether, $|B_i\cap P_j|=m^2$. If $L$ is a division sudoku then $B_i\cap P_j$ intersects all $m^2$ columns of $X$ and it follows that every minicolumn of $B_i$ contains precisely one element of $P_j$. Conversely, if every minicolumn of $B_i$ contains precisely one element of $P_j$ then $B_i\cap P_j$ intersects all columns of $X$. The argument for minirows is analogous by considering a stack $S_i$ and a pile $P_j$.
\end{proof}

Condition (iv) of Proposition \ref{Pr:DSChar} is particularly suitable from a visual perspective since it does not require us to calculate and inspect conjugates.

We could develop a fully symmetric terminology in which minisquares, minirows and minicolumns would make sense not only for the intersections of bands and stacks but also for the intersections of bands and piles, and stacks and piles. Then condition (iv) could be equivalently restated with respect to the new kind of miniobjects.

The next proposition offers a quasi-equational characterization of division sudokus in the variety of quasigroups.

\begin{proposition}\label{Pr:DSQuasi}
Let $\{X_1,\dots,X_m\}$ be a sudoku partition of $X$. Write $x\sim y$ if $x$ and $y$ belong to the same block of $X$. Let $Q=(X,\cdot,\rdiv,\ldiv)$ be a quasigroup. Consider the conditions
\begin{align}
    &\text{for all }x,\,y,\,u,\,v\in X\text{ if } x\sim u,\,y\sim v\text{ and } x \cdot y=u \cdot v\text{ then } (x,y)=(u,v),\label{Eq:DSQuasi1}\\
    &\text{for all }x,\,y,\,u,\,v\in X\text{ if } x\sim u,\,y\sim v\text{ and } x\rdiv y = u\rdiv v\text{ then } (x,y)=(u,v),\label{Eq:DSQuasi2}\\
    &\text{for all }x,\,y,\,u,\,v\in X\text{ if } x\sim u,\,y\sim v\text{ and } x\ldiv y = u\ldiv v\text{ then } (x,y)=(u,v).\label{Eq:DSQuasi3}
\end{align}
Then $Q$ is a sudoku if and only if \eqref{Eq:DSQuasi1} holds. Furthermore, $Q$ is a division sudoku if and only if all \eqref{Eq:DSQuasi1}, \eqref{Eq:DSQuasi2} and \eqref{Eq:DSQuasi3} hold.
\end{proposition}
\begin{proof}
We have $x\sim u$ and $y\sim v$ if and only if $(x,y,x\cdot y)$, $(u,v,u\cdot v)$ belong to the same minisquare, say $M$. The implication
\begin{displaymath}
    x\cdot y=u\cdot v\implies (x,y)=(u,v)
\end{displaymath}
then says that no symbol appears twice in $M$. Equivalently, every symbol appears in $M$, which is the sudoku condition by definition. The rest follows from Proposition \ref{Pr:DSChar} and the fact that $Q^{(1,3)} = (Q,\rdiv)$, $Q^{(2,3)} = (Q,\ldiv)$.
\end{proof}

For a quasigroup $(X,\cdot,\rdiv,\ldiv)$ and $x\in X$ consider the three bijections $L_x$, $R_x$, $D_x\in\sym{X}$ defined by
\begin{displaymath}
    L_x(y) = x\cdot y,\quad R_x(y) = y\cdot x,\quad D_x(y) = x\rdiv y.
\end{displaymath}
The translations $L_x$, $R_x$ are well-known and the bijections $D_x$ have been studied in \cite{Belousov, StanovskyVojtechovsky}, for instance.

\begin{definition}
Given an equivalence relation $\sim$ on $X$ and $\theta\in\sym{X}$, we say that $\theta$ \emph{shreds} $\sim$ if whenever $u\sim v$ and $u\ne v$ then $\theta(u)\not\sim\theta(v)$. Equivalently, $\theta$ shreds $\sim$ if whenever $u\sim v$ and $\theta(u)\sim\theta(v)$ then $u=v$.
\end{definition}

Note that if $\{X_1,\dots,X_m\}$ is a sudoku partition of $X$ and $\sim$ is the corresponding equivalence relation, then $\theta\in\sym{X}$ shreds $\sim$ if and only if $|\theta(X_i)\cap X_j|=1$ for every $1\le i$, $j\le m$.

\begin{proposition}\label{Pr:DSPerms}
Let $\{X_1,\dots,X_m\}$ be a sudoku partition of $X$ and let $\sim$ be the corresponding equivalence relation. Let $Q=(X,\cdot,\rdiv,\ldiv)$ be a quasigroup. Consider the conditions
\begin{align}
    &\text{for all }x,\,u,\,v\in X\text{ if }u\sim v\text{ and }x\rdiv u\sim x\rdiv v\text{ then }u=v,\label{Eq:DSPerms1}\\
    &\text{for all }x,\,u,\,v\in X\text{ if }u\sim v\text{ and }x\cdot u\sim x\cdot v\text{ then }u=v,\label{Eq:DSPerms2}\\
    &\text{for all }x,\,u,\,v\in X\text{ if }u\sim v\text{ and }u\cdot x\sim v\cdot x\text{ then }u=v.\label{Eq:DSPerms3}
\end{align}
Then $Q$ is a sudoku if and only if \eqref{Eq:DSPerms1} holds, which happens if and only if for every $x\in X$ the bijection $D_x$ shreds $\sim$. Furthermore, $Q$ is a division sudoku if and only if all \eqref{Eq:DSPerms1}, \eqref{Eq:DSPerms2}, \eqref{Eq:DSPerms3} hold, which happens if and only if for every $x\in X$ the bijections $D_x$, $L_x$, $R_x$ shred $\sim$.
\end{proposition}
\begin{proof}
Upon substituting $x\rdiv y$ for $x$ and $u\rdiv v$ for $u$ into \eqref{Eq:DSQuasi1}, we obtain
\begin{displaymath}
    \text{for all }x,\,y,\,u,\,v\in X\text{ if }x\rdiv y\sim u\rdiv v,\,y\sim v\text{ and } x=u\text{ then }(x\rdiv y,y)=(u\rdiv v,v).
\end{displaymath}
Using $x=u$ and eliminating the variable $u$, we get
\begin{displaymath}
    \text{for all }x,\,y,\,v\in X\text{ if }x\rdiv y\sim x\rdiv v\text{ and }y\sim v\text{ then }y=v.
\end{displaymath}
This is \eqref{Eq:DSPerms1} up to renaming of variables, and it says that every $D_x$ shreds $\sim$.

Similarly, upon substituting $x\cdot y$ for $x$ and $u\cdot v$ for $u$ into \eqref{Eq:DSQuasi2}, we get
\begin{displaymath}
    \text{for all }x,\,y,\,u,\,v\in X\text{ if }x \cdot y\sim u\cdot v,\,y\sim v\text{ and }x=u\text{ then }(x\cdot y,y)=(u\cdot v,v),
\end{displaymath}
which can be rewritten as
\begin{displaymath}
    \text{for all }x,\,y,\,v\in X\text{ if }x \cdot y\sim x\cdot v\text{ and }y\sim v\text{ then }y=v.
\end{displaymath}
This is \eqref{Eq:DSPerms2}, which says that every $L_x$ shreds $\sim$.

Finally, upon substituting $x\cdot y$ for $y$ and $u\cdot v$ for $v$ into \eqref{Eq:DSQuasi3}, we get
\begin{displaymath}
    \text{for all }x,\,y,\,u,\,v\in X\text{ if }x\sim u,\,x \cdot y\sim u \cdot v\text{ and }y=v\text{ then }(x,x\cdot y)=(u,u\cdot v),
\end{displaymath}
which can be rewritten as
\begin{displaymath}
    \text{for all }x,\,y,\,u\in X\text{ if }x\sim u\text{ and }x \cdot y\sim u \cdot y\text{ then }x=u.
\end{displaymath}
This is \eqref{Eq:DSPerms3}, which says that every $R_x$ shreds $\sim$.
\end{proof}

\section{Invariants of division sudokus}\label{Sc:Invariants}

In this section we introduce powerful invariants of division sudokus and enumerate division sudokus of rank $3$. Our intent is to use only invariants that can be calculated by hand with a reasonable effort. We again start with two notions for classifying division sudokus while preserving their combinatorial properties.

\begin{definition}\label{Df:DsIsotopic}
Let $\{X_1,\dots,X_m\}$ be a sudoku partition of $X$. Two division sudokus $L_1$, $L_2$ on $X$ are \emph{ds-isotopic} (short for \emph{division sudoku-isotopic}) if there exists an isotopism $(\alpha,\beta,\gamma)$ from $L_1$ to $L_2$ such that $\alpha$, $\beta$, $\gamma$ are partition preserving. Two division sudokus $L_1$, $L_2$ are \emph{ds-paratopic} (short for \emph{division sudoku-paratopic}) if $L_1$ is ds-isotopic to a conjugate of $L_2$. The equivalence classes of division sudokus induced by ds-paratopy are called \emph{main ds-classes}.
\end{definition}

An \emph{intercalate} in a latin square is a $2\times 2$ latin subsquare. Obviously, a given intercalate is contained in one or two bands, and similarly for stacks and piles.

\begin{lemma}\label{Lm:IntercalateStructure}
Let $I$ be an intercalate of a division sudoku. Then at most one of the following situations can occur:
\begin{enumerate}
\item[(i)] $I$ is contained in a band,
\item[(ii)] $I$ is contained in a stack,
\item[(iii)] $I$ is contained in a pile.
\end{enumerate}
\begin{proof}
Suppose that both (i) and (ii) hold. Let $I$ be contained in a band $B_i$ and stack $S_j$. Let $a$ be one of the two symbols contained in $I$. Then $a$ occurs at least twice in the minisquare $B_i\cap S_j$, a contradiction.

We proceed similarly if any two conditions from among (i), (ii), (iii) hold, using the fact that the concept of division sudokus is invariant under conjugacy.
\end{proof}
\end{lemma}

When $L$ is a division sudoku of rank $3$ and $I$ is an intercalate of $L$, we define $\mathrm{band}_L(I)$ (resp. $\mathrm{stack}_L(I)$, $\mathrm{pile}_L(I)$) to be the unique band (resp. stack, pile) of $L$ that either contains $I$ or is disjoint from $I$. Note that rank $3$ is necessary for this definition to make sense.

\begin{definition}
Let $L$ be a standard division sudoku of rank $3$. The \emph{intercalate structure invariant} $\iota_L$ of $L$ is a $\{0,1\}$-array indexed by bands, stacks and piles, where
$\iota_L(B_i,S_j,P_k) = 1$ if and only if there exists an intercalate $I$ of $L$ such that $(B_i,S_j,P_k) = (\mathrm{band}_L(I),\mathrm{stack}_L(I),\mathrm{pile}_L(I))$.

Two intercalate structure invariants are said to be \emph{equivalent} if one is obtained from the other by a permutation of band coordinates (rows), stack coordinates (columns) and pile coordinates (symbols).
\end{definition}

The array $\iota_L$ can be visualized as a hypergraph with vertices consisting of bands, stacks and piles, and hyperedges formed by those triangles $(B_i,S_j,P_k)$ for which $\iota_L(B_i,S_j,P_k)=1$. These graphs become unwieldy when $L$ is rich in intercalates.

\begin{example}
Below is a standard division sudoku and its intercalate structure invariant.

\medskip
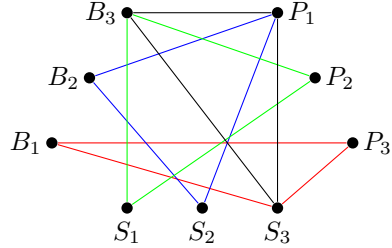
\begin{center} 
\begin{tabular}{P{55mm}P{55mm}}
\sudoku{\cc1&4&7&\cc2&5&8&3&9&6}{8&2&5&9&3&6&4&1&7}{6&9&3&4&7&1&8&5&2}{\cc2&5&8&\cc1&6&9&7&3&4}{9&3&6&7&2&4&5&8&1}{4&7&1&5&8&3&2&6&9}{3&8&4&6&9&2&1&7&5}{5&1&9&3&4&7&6&2&8}{7&6&2&8&1&5&9&4&3}
&
\raisebox{-17mm}{
\begin{tikzpicture}[scale=0.5,every node/.style={circle,fill,inner sep=0pt,minimum size=1.5mm}]
    \node (S1) at (-2, -3.464) [label=below:$S_1$] {};
    \node (S2) at (0, -3.464) [label=below:$S_2$] {};
    \node (S3) at (2, -3.464) [label=below:$S_3$] {};
    \node (B1) at (-4, -1.732) [label=left:$B_1$] {};
    \node (B2) at (-3, 0) [label=left:$B_2$] {};
    \node (B3) at (-2, 1.732) [label=left:$B_3$] {};
    \node (P3) at (4, -1.732) [label=right:$P_3$] {};
    \node (P2) at (3, 0) [label=right:$P_2$] {};
    \node (P1) at (2, 1.732) [label=right:$P_1$] {};
    \draw [red]    (B1) -- (S3) -- (P3) -- (B1);
    \draw [blue]   (B2) -- (S2) -- (P1) -- (B2);
    \draw [green]  (B3) -- (S1) -- (P2) -- (B3);
    \draw [black]  (B3) -- (S3) -- (P1) -- (B3);
\end{tikzpicture}
}
\end{tabular}
\end{center}
\medskip

\noindent We have colored the hyperedges to improve legibility. For instance, the black hyperedge $(B_3,S_3,P_1)$ is there on account of the highlighted intercalate that avoids $B_3$, avoids $S_3$, and lies within $P_1$.
\end{example}

\begin{proposition}
If $L_1$, $L_2$ are ds-isotopic standard division sudokus of rank $3$ then $\iota_{L_1}$, $\iota_{L_2}$ are equivalent intercalate structure invariants.
\end{proposition}
\begin{proof}
Let $(\alpha,\beta,\gamma)$ be a ds-isotopism from $L_1$ to $L_2$. Let $I_1$ be an intercalate of $L_1$ and let $I_2$ be the image of $I_1$ under $(\alpha,\beta,\gamma)$. If $I_1$ belongs to a band of $L_1$ then $I_2$ belongs to a band of $L_2$ since $\alpha$ is partition preserving. We can argue similarly for stacks and piles. 
\end{proof}

\begin{remark}
It is possible to strengthen the intercalate structure invariant in various ways. For instance, the entry $\iota_L(B_i,S_j,P_k)$ could \emph{count} the number of intercalates $I$ of $L$ such that $(B_i,S_j,P_k) = (\mathrm{band}_L(I),\mathrm{stack}_L(I),\mathrm{pile}_L(I))$, rather than just indicate whether such an intercalate exists. Alternatively, taking Lemma \ref{Lm:IntercalateStructure} into account, we could define $\mathrm{type}_L(I)$ by
\begin{displaymath}
    \mathrm{type}_L(I)=\left\{\begin{array}{ll}
        1,&\text{ if $I$ is contained in a band of $L$},\\
        2,&\text{ if $I$ is contained in a stack of $L$},\\
        3,&\text{ if $I$ is contained in a pile of $L$},\\
        4,&\text{ otherwise},
    \end{array}\right.
\end{displaymath}
and consider a $3\times 3\times 3\times 4$ array, where a given intercalate $I$ contributes $1$ to the cell indexed by $(\mathrm{band}_L(I),\mathrm{stack}_L(I),\mathrm{pile}_L(I),\mathrm{type}_L(I))$. It turns out that we do not need these refined invariants for our purposes.
\end{remark}


\begin{definition}
Let $L$ be a standard division sudoku of rank $3$. The \emph{minisquare structure invariant} $\mu_L$ of $L$ is a tripartite digraph obtained as follows. The nine vertices of $\mu_L$ are the bands, stacks and piles of $L$. Let $B$ be a band and $S$ a stack of $L$, and let $S_i$, $S_j$ be the two stacks different from $S$. For $1\le k\le 3$ let $C_{i,k}$ (resp. $C_{j,k}$) be the set of symbols in the $k$th minicolumn of $B\cap S_i$ (resp. $B\cap S_j$). Then $B\to S$ is a directed edge if and only if $\{C_{i,1},C_{i,2},C_{i,3}\} = \{C_{j,1},C_{j,2},C_{j,3}\}$. Directed edges between stacks and bands, bands and piles, piles and bands, stacks and piles, and piles and stacks are determined analogously using conjugacy.

Two minisquare structure invariants are said to be \emph{equivalent} if one is obtained from the other by a permutation of band vertices, stack vertices and pile vertices.
\end{definition}

\begin{example}
Below is a standard division sudoku and its minisquare structure invariant.

\medskip
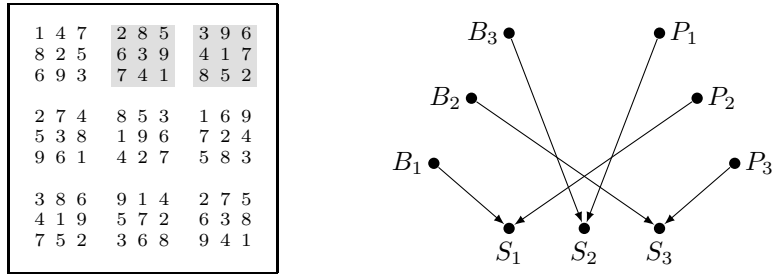
\begin{center} 
\begin{tabular}{P{55mm}P{55mm}}
\sudoku{1&4&7&\cc2&\cc8&\cc5&\cc3&\cc9&\cc6}{8&2&5&\cc6&\cc3&\cc9&\cc4&\cc1&\cc7}{6&9&3&\cc7&\cc4&\cc1&\cc8&\cc5&\cc2}{2&7&4&8&5&3&1&6&9}{5&3&8&1&9&6&7&2&4}{9&6&1&4&2&7&5&8&3}{3&8&6&9&1&4&2&7&5}{4&1&9&5&7&2&6&3&8}{7&5&2&3&6&8&9&4&1}
&
\raisebox{-17mm}{
\begin{tikzpicture}[scale=0.5,every node/.style={circle,fill,inner sep=0pt,minimum size=1.5mm}]
    \node (S1) at (-2, -3.464) [label=below:$S_1$] {};
    \node (S2) at (0, -3.464) [label=below:$S_2$] {};
    \node (S3) at (2, -3.464) [label=below:$S_3$] {};
    \node (B1) at (-4, -1.732) [label=left:$B_1$] {};
    \node (B2) at (-3, 0) [label=left:$B_2$] {};
    \node (B3) at (-2, 1.732) [label=left:$B_3$] {};
    \node (P3) at (4, -1.732) [label=right:$P_3$] {};
    \node (P2) at (3, 0) [label=right:$P_2$] {};
    \node (P1) at (2, 1.732) [label=right:$P_1$] {};
    \draw [-latex] (B1)--(S1);
    \draw [-latex] (B2)--(S3);
    \draw [-latex] (B3)--(S2);
    \draw [-latex] (P1)--(S2);
    \draw [-latex] (P2)--(S1);
    \draw [-latex] (P3)--(S3);
\end{tikzpicture}
}
\end{tabular}
\end{center}
\medskip

\noindent For instance, there is a directed edge $B_1\to S_1$ because in the two highlighted minisquares $B_1\cap S_2$, $B_1\cap S_3$, the minicolumns contain the same sets of symbols, namely $\{1,5,9\}$, $\{2,6,7\}$ and $\{3,4,8\}$.
\end{example}

We immediately see:

\begin{proposition}
If $L_1$, $L_2$ are ds-isotopic standard division sudokus of rank $3$ then $\mu_{L_1}$, $\mu_{L_2}$ are equivalent minisquare structure invariants
\end{proposition}

\section{Enumeration}\label{Sc:Enumeration}

We enumerate division sudokus of rank $3$.

\subsection{Partial isotopisms, s-isotopisms and ds-isotopisms}

\begin{lemma}
Let $Q_1=(X,\cdot,\rdiv,\ldiv)$, $Q_2=(X,\circ,\drdiv,\dldiv)$ be quasigroups and let $(\alpha,\beta,\gamma)$ be an isotopism from $Q_1$ onto $Q_2$. If $A$, $B\subseteq X$ are such that $A\cdot B =\{a\cdot b:a\in A,\,b\in B\} = X$ then $(\alpha,\beta,\gamma)$ is determined by the values $(\alpha(a):a\in A)$ and $(\beta(b):b\in B)$.
\end{lemma}
\begin{proof}
Suppose that $\alpha$ is known on $A$ and $\beta$ is known on $B$. Then $\gamma(a\cdot b) = \alpha(a)\circ\beta(b)$ shows that $\gamma$ is known on $A\cdot B=X$. For any $b\in B$ and $x\in X$ we have $\alpha(x) = \gamma(x\cdot b)\drdiv \beta(b)$ and thus $\alpha$ is known on $X$. Similarly, for any $a\in A$ and $y\in X$ we have $\beta(y) = \alpha(a)\dldiv\gamma(a\cdot y)$, so $\beta$ is also known on $X$.
\end{proof}

\begin{corollary}\label{Cr:ItpData}
Let $\{X_1,\dots,X_m\}$ be a sudoku partition of $X$. Let $Q_1$ be a sudoku on $X$ and $Q_2$ a quasigroup. Then for any $1\le i$, $j\le m$, an isotopism $(\alpha,\beta,\gamma)$ from $Q_1$ to $Q_2$ is determined by the values $(\alpha(x):x\in X_i)$ and $(\beta(x):x\in X_j)$.
\end{corollary}

\begin{lemma}\label{Lm:SudItpData}
Let $\{X_1,\dots,X_m\}$ be a sudoku partition of $X$ and let $Q_1 = (X,\cdot)$, $Q_2 =(X,\circ)$ be sudokus. Then for any $1\le i$, $j\le m$ and any $(m-1)$-element subsets $X_i'\subseteq X_i$ and $X_j'\subseteq X_j$, an s-isotopism $(\alpha,\beta,\gamma)$ from $Q_1$ to $Q_2$ is determined by the values $(\alpha(x):x\in X_i')$ and $(\beta(x):x\in X_j')$.
\end{lemma}
\begin{proof}
Since $\alpha$ maps blocks to blocks, it maps $X_i$ to some block $X_k$. We have $X_i = X_i'\cup\{x_i\}$ and $X_k = \alpha(X_i')\cup\{x_k\}$ for some $x_i\in X_i$, $x_k\in X_k$, and then $\alpha(x_i)=x_k$ is forced. Similarly for $\beta$ and we are done by Corollary \ref{Cr:ItpData}.
\end{proof}

Note that in the above three results we have assumed that the quasigroup operations for $Q_1$, $Q_2$ are fully known. However, the results can be sharpened by specifying exactly which parts of the multiplication and division tables of $Q_1$ and $Q_2$ need to be known. Lemma \ref{Lm:DataForPartialDS} below is an example of such a result. It is based on the partial division sudoku in Figure \ref{Fg:PartialDivisionSudoku}. We shall see later that every standard division sudoku of rank $3$ is ds-isotopic to a standard division sudoku that extends the partial table in Figure \ref{Fg:PartialDivisionSudoku}.

\begin{lemma}\label{Lm:DataForPartialDS}
Let $\{X_1,X_2,X_3\}$ be a sudoku partition of $X$ and let $Q_1=(X,\cdot)$, $Q_2=(X,\circ)$ be standard division sudokus whose tables extend the partial division sudoku table of Figure \emph{\ref{Fg:PartialDivisionSudoku}}. Then a ds-isotopism $(\alpha,\beta,\gamma)$ from $Q_1$ to $Q_2$ is determined by the values $\alpha^{-1}(1)$, $\alpha^{-1}(2)$, $\beta^{-1}(1)$, $\beta^{-1}(2)$, by the multiplication table of $Q_1$ and by the partial multiplication table of $Q_2$ given in Figure \emph{\ref{Fg:PartialDivisionSudoku}}.
\end{lemma}
\begin{proof}
Let $\rdiv$, $\ldiv$ be the division operations in $Q_1$, which are certainly determined by the multiplication in $Q_1$. The permutation $\alpha$ maps blocks to blocks and hence the knowledge of $\alpha^{-1}(1)$, $\alpha^{-1}(2)$ determines $\alpha^{-1}(3)$. The entry $\beta^{-1}(3)$ is determined similarly from $\beta^{-1}(1)$, $\beta^{-1}(2)$. We have $\alpha(x)\circ \beta(y) = \gamma(x\cdot y)$ and thus $\gamma^{-1}(x\circ y) = \alpha^{-1}(x)\cdot\beta^{-1}(y)$ for every $x$, $y$. With $x$, $y\in \{1,2,3\}$, the values $\alpha^{-1}(x)\cdot\beta^{-1}(y)$ and $x\circ y$ are known and therefore determine $\gamma^{-1}(z)$ for every $z$, using the sudoku property in the top left minisquare of $Q_2$. We can now determine $\beta$ by focusing on the rest of the top band of $Q_2$. For every $4\le y\le 9$ there is a unique cell $(x,y)$ with $1\le x\le 3$ for which $x\circ y$ is given in the partial table. For instance, $1\circ 4=2$, so $\alpha^{-1}(1)\cdot\beta^{-1}(4) = \gamma^{-1}(2)$, which yields $\beta^{-1}(4) = \alpha^{-1}(1)\ldiv \gamma^{-1}(2)$. The permutation $\alpha$ is determined dually by focusing on the left stack of $Q_2$.
\end{proof}

\subsection{The enumeration}

\begin{figure}[ht]
\begin{center}
\def\cc{$\cdot$}
\sudoku{1&4&7&2&\cc&\cc&3&\cc&\cc}{8&2&5&\cc&3&\cc&\cc&1&\cc}{6&9&3&\cc&\cc&1&\cc&\cc&2}{2&\cc&\cc&\cc&\cc&\cc&\cc&\cc&\cc}{\cc&3&\cc&\cc&\cc&\cc&\cc&\cc&\cc}{\cc&\cc&1&\cc&\cc&\cc&\cc&\cc&\cc}{3&\cc&\cc&\cc&\cc&\cc&\cc&\cc&\cc}{\cc&1&\cc&\cc&\cc&\cc&\cc&\cc&\cc}{\cc&\cc&2&\cc&\cc&\cc&\cc&\cc&\cc}
\end{center}
\caption{Partial table of every standard division sudoku of rank $3$ up to ds-isotopism.}
\label{Fg:PartialDivisionSudoku}
\end{figure}

\begin{lemma}\label{Lm:Template}
Every standard division sudoku of rank $3$ is ds-isotopic to a standard division sudoku representative that extends the partial table in Figure \emph{\ref{Fg:PartialDivisionSudoku}}. The ds-isotopism class of each of these representatives consists of $2^{11}\cdot 3^8 =13436928$ standard division sudokus.
\end{lemma}
\begin{proof}
Using ds-isotopisms, we will bring a given standard division sudoku into the desired form while keeping track of the number of division sudokus identified in the process.

First consider the minisquare $M_{1,1}=B_1\cap S_1$. Either the symbols on the main diagonal of $M_{1,1}$ form a block, or the symbols on the main antidiagonal of $M_{1,1}$ form a block. By permuting the $3$ blocks of symbols and then permuting within each of the blocks of symbols, we can therefore assume that $M_{1,1}$ takes on one of the following two forms:
\begin{displaymath}
\begin{array}{ccc}
    1&4&7\\
    8&2&5\\
    6&9&3
\end{array}
\quad\text{or}\quad
\begin{array}{ccc}
    1&7&4\\
    8&5&2\\
    6&3&9
\end{array}.
\end{displaymath}
We have identified $|\sym{3}|^4 = 6^4$ division sudokus so far. By applying the column permutation $(2,3)$ to the second form, we obtain the first form, contributing another factor of $|C_2|=2$.

By permuting the columns $4$, $5$, $6$ we can assume that the symbols on the main diagonal of the minisquare $M_{1,2}=B_1\cap S_2$ form the block $\{1,2,3\}$. By permuting the columns $7$, $8$, $9$ we can assume that the symbols on the main diagonal of the minisquare $M_{1,3}=B_1\cap S_3$ form the block $\{1,2,3\}$. By exchanging the column blocks $\{4,5,6\}$ and $\{7,8,9\}$, if necessary, we can assume that $B_1$ is as in Figure \ref{Fg:PartialDivisionSudoku}, taking advantage of the latin property in the top $3$ rows. The factor here is $|\sym{3}|^2\cdot |C_2| = 6^2\cdot 2$.

We proceed analogously in the stack $S_1$, encountering another factor of $|\sym{3}|^2\cdot |C_2| = 6^2\cdot 2$.
\end{proof}

\begin{lemma}[\Keyboard]\label{Lm:Mace4}
There are $7741$ standard division sudokus extending the partial table in Figure \emph{\ref{Fg:PartialDivisionSudoku}}.
\end{lemma}
\begin{proof}
This can be shown in a few seconds with a finite model builder. We used \texttt{Mace4} \cite{Mace4} with the following input file to search for all models of size $9$. The predicate \texttt{B} keeps track of the blocks of the sudoku partition. Note that \texttt{Mace4} labels elements in models of size $n$ with $\{0,\dots,n-1\}$. The division sudoku conditions in the input file correspond to those of Proposition \ref{Pr:DSQuasi}.
\begin{small}
\begin{verbatim}
% latin square (quasigroup)
x*(x\y) = y.    x\(x*y) = y.    (x*y)/y = x.    (x/y)*y = x.
% predicate B defines standard blocks
B(0)=0. B(1)=0. B(2)=0. B(3)=1. B(4)=1. B(5)=1. B(6)=2. B(7)=2. B(8)=2.
% divison sudoku conditions
all x all y all u all v ((B(x)=B(u) & B(y)=B(v) & x*y = u*v)->(x=u & y=v)).
all x all y all u all v ((B(x)=B(u) & B(y)=B(v) & x/y = u/v)->(x=u & y=v)).
all x all y all u all v ((B(x)=B(u) & B(y)=B(v) & x\y = u\v)->(x=u & y=v)).
% partial table
0*0=0. 0*1=3. 0*2=6. 1*0=7. 1*1=1. 1*2=4. 2*0=5. 2*1=8. 2*2=2.
0*3=1. 1*4=2. 2*5=0. 0*6=2. 1*7=0. 2*8=1.
3*0=1. 4*1=2. 5*2=0. 6*0=2. 7*1=0. 8*2=1.
\end{verbatim}
\end{small}
\end{proof}

Since $7741$ happens to be a prime number, we deduce that it is not possible to further identify division sudokus in the proof of Lemma \ref{Lm:Template} without splitting the argument into cases.

\begin{theorem}[\Keyboard]
There are $104015259648 = 2^{11}\cdot 3^8\cdot 7741$ standard division sudokus of rank $3$. There are $186$ standard division sudokus of rank $3$ up to ds-isotopism. There are $45$ standard division sudokus of rank $3$ up to ds-paratopism.
\end{theorem}
\begin{proof}
The first result follows from Lemmas \ref{Lm:Template} and \ref{Lm:Mace4}.

The $7741$ representatives from Lemma \ref{Lm:Mace4} must be filtered up to ds-isotopism. In general, filtering thousands of quasigroups of order $9$ up to isotopism is a nontrivial problem. Even the case of ds-isotopisms is hard, since there are $1296=6^4$ permutations in $\sym{9}$ that preserve the standard partition. We could take two of the $7741$ quasigroups at a time and check for a ds-isotopism using Lemma \ref{Lm:SudItpData}. Each of these checks is fast since only the parameters $\alpha(1)$, $\alpha(2)$, $\beta(1)$, $\beta(2)$ need to be given, but there are $\binom{7741}{2}$ checks to be conducted.

Taking advantage of Lemma \ref{Lm:DataForPartialDS} is much faster. Let $Q_1 = (X,\cdot)$ be one of the $7741$ division sudokus. We choose $\alpha^{-1}(1)$ arbitrarily and then $\alpha^{-1}(2)$ in the same block as $\alpha^{-1}(1)$. Similarly for $\beta^{-1}(1)$ and $\beta^{-1}(2)$. We then complete $\alpha$, $\beta$ and $\gamma$ from the knowledge of $Q_1$ and the partial table of $Q_2$ as in the proof of Lemma \ref{Lm:DataForPartialDS}. If the resulting maps are partition preserving permutations, we set $x\circ y = \gamma(\alpha^{-1}(x)\cdot\beta^{-1}(y))$, hence constructing the unique quasigroup $Q_2=(X,\circ)$ such that $(\alpha,\beta,\gamma)$ is a ds-isotopism from $Q_1$ to $Q_2$. If $Q_2$ is among the $7741$ representatives, we add it to the ds-isotopism class of $Q_1$. We implemented this algorithm in \texttt{GAP}. In $3$ seconds, the algorithm returns $186$ representatives with classes of sizes $1$ (once), $3$ ($9$ times), $6$ ($3$ times), $9$ ($22$ times), $18$ ($7$ times), $27$ ($15$ times) and $54$ ($129$ times).

Finally, for main ds-classes, we first subdivide the $186$ representatives according to the intercalate structure invariant, cf. Section \ref{Sc:Invariants}. Then, given a subset of representatives with the same invariant, we check if any two lie in the same main ds-class by generating all $6$ conjugates and checking for a ds-isotopism. Note that the conjugates do not necessarily have the form of Figure \ref{Fg:PartialDivisionSudoku} and we therefore cannot use Lemma \ref{Lm:DataForPartialDS}, but Lemma \ref{Lm:SudItpData} applies. We obtained $45$ representatives in about $1$ minute.
\end{proof}

The $186$ standard division sudokus of rank $3$ up to ds-isotopism are given in the Appendix, labeled $\DS(9,1)$, $\dots$, $\DS(9,186)$. These are the pairwise structurally different division sudokus of rank $3$ when the roles of rows, columns and symbols are fixed.

\begin{remark}
Each division sudoku $\DS(9,i)$ is ds-isotopic to an idempotent quasigroup, that is, a quasigroup in which the identity $x\cdot x=x$ holds. (It suffices to permute the rows and columns in the second block and the rows and columns in the third block.) For convenience, we will sometimes work with idempotent versions of the division sudokus $\DS(9,i)$.
\end{remark}

\begin{table}[ht]
\caption{The main ds-classes of standard division sudokus of rank $3$. Integers $i$, $j$ are listed in the same set if and only if $\DS(9,i)$, $\DS(9,j)$ belong to the same main ds-class.}\label{Tb:MainDSClasses}
\begin{center}
\begin{small}
\begin{tabular}{lll}
\{1,\,53,\,135\}&                         \{2,\,35,\,156\}&                         \{3,\,40,\,42,\,149,\,163,\,180\}\\
\{4\}&                                    \{5,\,24,\,46,\,64,\,123,\,126\}&         \{6,\,12,\,14,\,29,\,43,\,102\}\\
\{7,\,37,\,157\}&                         \{8,\,32,\,47,\,142,\,161,\,172\}&        \{9,\,10,\,49\}\\
\{11,\,31,\,52,\,134,\,159,\,171\}&       \{13,\,26,\,56,\,65,\,125,\,127\}&        \{15,\,39,\,58,\,73,\,98,\,101\}\\
\{16,\,55,\,150\}&                        \{17\}&                                   \{18,\,19,\,21,\,33,\,60,\,117\}\\
\{20,\,62,\,131\}&                        \{22,\,34,\,61,\,75,\,118,\,119\}&        \{23,\,36,\,63,\,132,\,158,\,170\}\\
\{25,\,28,\,67,\,100,\,111,\,113\}&       \{27,\,69,\,148\}&                        \{30,\,41,\,80,\,155,\,164,\,181\}\\
\{38,\,45,\,103\}&                        \{44,\,48,\,77,\,109,\,137,\,139\}&       \{50,\,54,\,133\}\\
\{51\}&                                   \{57,\,140,\,174\}&                       \{59,\,99,\,107\}\\
\{66,\,124,\,168\}&                       \{68\}&                                   \{70,\,130,\,165,\,166,\,173,\,177\}\\
\{71,\,122,\,162,\,169,\,184,\,185\}&     \{72,\,78,\,83,\,95,\,96,\,97\}&          \{74,\,82,\,85,\,104,\,138,\,145\}\\
\{76,\,120\}&                             \{79,\,84,\,93,\,116,\,151,\,152\}&       \{81,\,115,\,146\}\\
\{86,\,88,\,106,\,112,\,143,\,147\}&      \{87,\,144,\,178\}&                       \{89,\,92,\,110,\,114,\,141,\,154\}\\
\{90,\,153,\,182\}&                       \{91,\,105,\,121,\,129,\,160,\,186\}&     \{94,\,108,\,128,\,136,\,167,\,176\}\\
\{175\}&                                  \{179\}&                                  \{183\}
\end{tabular}
\end{small}
\end{center}
\end{table}

Table \ref{Tb:MainDSClasses} lists the $45$ main ds-classes of standard division sudokus of rank $3$. These are the pairwise structurally different division sudokus of rank $3$ when the roles of rows, columns and symbols are allowed to be permuted.

Finally, it is reasonable to ask which of the $186$ standard division sudokus of rank $3$ up to ds-isotopism are isotopic as quasigroups. Using Lemma \ref{Lm:SudItpData}, several hours of computing time yield:

\begin{proposition}[\Keyboard]
There are $183$ standard division sudokus of rank $3$ up to isotopism. Among the $186$ standard division sudokus of rank $3$ up to ds-isotopism, only the following pairs are isotopic: $\DS(9,18)$ and $\DS(9,19)$, $\DS(9,21)$ and $\DS(9,60)$, $\DS(9,33)$ and $\DS(9,117)$.
\end{proposition}

\subsection{The distinguishing power of the invariants}

The two invariants introduced in Section \ref{Sc:Invariants} are quite powerful. By calculating the invariants for the $186$ standard division sudokus of rank $3$ up to ds-isotopism we obtain:

\begin{proposition}[\Keyboard] The intercalate structure invariant $\iota$ attains precisely $148$ distinct values on the set of standard division sudokus of rank $3$. The minisquare structure invariant $\mu$ attains precisely $139$ distinct values on the set of standard division sudokus of rank $3$. The combined invariant $(\iota,\mu)$ attains precisely $183$ distinct values on the set of standard division sudokus of rank $3$.
\end{proposition}



Only the following pairs of division sudokus cannot be distinguished by the two invariants:
\begin{displaymath}
    \DS(9{,}91)\text{ and }\DS(9{,}105),\,\DS(9{,}121)\text{ and }\DS(9{,}129),\,\DS(9{,}160)\text{ and }\DS(9{,}186).
\end{displaymath}
These $6$ division sudokus belong to the same main ds-class. Let us have a closer look at them.

We start with the division sudoku $(Q,\cdot)$ below which is ds-isotopic to $\DS(9,121)$. For the convenience of the reader, we have also displayed the right division $\rdiv$ and the left division $\ldiv$ for $(Q,\cdot)$.

\medskip
\begin{center}
\begin{tabular}{P{35mm}P{35mm}P{35mm}}
\oplabel{$\cdot$}
\rowlabels{1}{2}{3}{4}{5}{6}{7}{8}{9}
\collabels{1}{2}{3}{4}{5}{6}{7}{8}{9}
\smalllabeledsudoku{1&6&7&8&3&4&9&2&5}{8&2&4&5&9&1&6&7&3}{5&9&3&2&6&7&1&4&8}{9&1&5&4&7&2&8&3&6}{6&7&2&3&5&8&4&9&1}{3&4&8&9&1&6&2&5&7}{2&5&9&6&8&3&7&1&4}{7&3&6&1&4&9&5&8&2}{4&8&1&7&2&5&3&6&9}
&
\oplabel{$\rdiv$}
\rowlabels{1}{2}{3}{4}{5}{6}{7}{8}{9}
\collabels{1}{2}{3}{4}{5}{6}{7}{8}{9}
\smalllabeledsudoku{1&4&9&8&6&2&3&7&5}{7&2&5&3&9&4&6&1&8}{6&8&3&5&1&7&9&4&2}{9&6&2&4&8&1&5&3&7}{3&7&4&2&5&9&8&6&1}{5&1&8&7&3&6&2&9&4}{8&5&1&9&4&3&7&2&6}{2&9&6&1&7&5&4&8&3}{4&3&7&6&2&8&1&5&9}
&
\oplabel{$\ldiv$}
\rowlabels{1}{2}{3}{4}{5}{6}{7}{8}{9}
\collabels{1}{2}{3}{4}{5}{6}{7}{8}{9}
\smalllabeledsudoku{1&8&5&6&9&2&3&4&7}{6&2&9&3&4&7&8&1&5}{7&4&3&8&1&5&6&9&2}{2&6&8&4&3&9&5&7&1}{9&3&4&7&5&1&2&6&8}{5&7&1&2&8&6&9&3&4}{8&1&6&9&2&4&7&5&3}{4&9&2&5&7&3&1&8&6}{3&5&7&1&6&8&4&2&9}
\end{tabular}
\end{center}
\medskip

The following result can be verified by inspection of the multiplication table $(Q,\cdot)$. The main diagonal of the table will be called just \emph{diagonal}, and a minisquare intersecting the diagonal will be called a \emph{diagonal minisquare}.

\begin{lemma}\label{Lm:1}
The quasigroup $(Q,\cdot)$ is idempotent and contains exactly 18 intercalates. Each intercalate contains exactly one cell from the diagonal. For each diagonal cell there are exactly 2 intercalates that include that cell. In each intercalate, the corner opposite from the diagonal cell lies in a diagonal minisquare. Each nondiagonal cell that is inside a diagonal minisquare is contained in exactly one intercalate.
\end{lemma}

\begin{lemma}\label{Lm:2}
Suppose that $(\alpha,\beta,\gamma)$ is an isotopism of $(Q,\cdot)$ onto either $(Q,\cdot)$ or the conjugate $(Q,\cdot)^{(1,2)}$. Then $(\alpha,\beta,\gamma)$ is a ds-isotopism and $\alpha=\beta=\gamma$.
\end{lemma}
\begin{proof}
The intercalate structure of $(Q,\cdot)$ implies that every isotopism sends a diagonal minisquare upon a diagonal minisquare. Hence every isotopism is a ds-isotopism. Since every isotopism sends intercalates onto intercalates and since the only cells with two intercalates through them are on the diagonal, we have $\alpha=\beta$. Using idempotency, we then have $\alpha(x) = \alpha(x)\cdot \alpha(x)= \alpha(x) \cdot \beta(x) = \gamma(x\cdot x) = \gamma(x)$ for each $x\in Q$. Note that the calculation goes through even if the target quasigroup is $(Q,\cdot)^{(1,2)}$ since $\alpha(x)\cdot\beta(x)=\beta(x)\cdot\alpha(x)$ here.
\end{proof}

\begin{lemma}
The quasigroup $(Q,\cdot)$ is not isotopic to $(Q,\cdot)^{(1,2)}$.
\end{lemma}
\begin{proof}
Suppose for a contradiction that $(Q,\cdot)$ is isotopic to $(Q,\cdot)^{(1,2)}$. Then the quasigroups are isomorphic by Lemma \ref{Lm:2}. But that is not possible since all left translations of $(Q,\cdot)$ contain a $4$-cycle, while no right translation of $(Q,\cdot)$ contains a $4$-cycle. (We remark that the automorphism group of $(Q,\cdot)$ acts transitively on $Q$.)
\end{proof}

By inspection of the three tables $(Q,\cdot)$, $(Q,\rdiv)$ and $(Q,\ldiv)$, we observe:

\begin{lemma}\label{Lm:4}
Let $\mu_{(Q,\cdot)}$ be the minisquare structure invariant of $(Q,\cdot)$. Then the only directed edges of $\mu_{(Q,\cdot)}$ are $P_i\to B_i$ and $P_i\to S_i$, for $1\le i\le 3$.
\end{lemma}

\begin{proposition}
The intercalate structure invariant and the minisquare structure invariant do not distinguish $(Q,\cdot)$ from $(Q,\cdot)^{(1,2)}$, $(Q,\rdiv)$ from $(Q,\ldiv)^{(1,2)}$, and $(Q,\ldiv)$ from $(Q,\rdiv)^{(1,2)}$.
\end{proposition}
\begin{proof}
By Lemma \ref{Lm:1}, the intercalate structure invariants of $(Q,\cdot)$ and $(Q,\cdot)^{(1,2)}$ coincide. By Lemma \ref{Lm:4}, the minisquare structure invariant of $(Q,\cdot)$ distinguishes piles from bands and stacks, but not bands from stacks. The rest follows by considering the conjugates of $(Q,\cdot)$.
\end{proof}

It can now be checked (by hand) that $(Q,\cdot)$ is indeed ds-isotopic to $\DS(9,121)$, $(Q,\cdot)^{(1,2)}$ is ds-isotopic to $\DS(9,129)$, $(Q,\rdiv)$ is ds-isotopic to $\DS(9,186)$, $(Q,\rdiv)^{(1,2)}$ is ds-isotopic to $\DS(9,105)$, $(Q,\ldiv)$ is ds-isotopic to $\DS(9,91)$, and $(Q,\ldiv)^{(1,2)}$ is ds-isotopic to $\DS(9,160)$.

\section{Division sudokus with multiple sudoku partitions}\label{Sc:MultiplePartitions}

\subsection{Sudoku tri-partitions}

So far we have used the same sudoku partition for rows, columns and symbols. We will now find it convenient to use three separate partitions.

\begin{definition}
A \emph{sudoku tri-partition} is an ordered triple $(X^B,X^S,X^P)$ of sudoku partitions $\{X^B_1,\dots,X^B_m\}$, $\{X^S_1,\dots,X^S_m\}$ and $\{X^P_1,\dots,X^P_m\}$ on the same underlying set $X$. Given a latin square $L$ on $X$, the \emph{bands}, \emph{stacks} and \emph{piles} are defined by
\begin{align*}
    B_i &= \{(x,y,z)\in O(L):x\in X^B_i\},\\
    S_i &= \{(x,y,z)\in O(L):y\in X^S_i\},\\
    P_i &= \{(x,y,z)\in O(L):z\in X^P_i\},
\end{align*}
respectively.

A latin square $L$ on $X$ is a \emph{sudoku with respect to a sudoku tri-partition} if the intersection of every band with every stack contains all symbols of $X$. A latin square $L$ on $X$ is a \emph{division sudoku with respect to a sudoku tri-partition} if the intersection of every band with every stack contains all symbols of $X$, the intersection of every band with every pile contains all columns of $X$, and the intersection of every stack with every pile contains all rows of $X$.
\end{definition}

Results from Section \ref{Sc:Char} remain valid in this more general setting when suitably interpreted. For instance, a latin square $L$ is a division sudoku with respect to the sudoku tri-partition $(X^B,X^S,X^P)$ if and only if the conjugates $L$, $L^{(1,3)}$, $L^{(2,3)}$ are sudokus with respect to the sudoku tri-partition $(X^B,X^S,X^P)$.

For a latin square $L$ on $X$, denote by $\pi(L)$ the number of sudoku tri-partitions with respect to which $L$ is a division sudoku. In this subsection we will determine $\pi(L)$ for every division sudoku $L$ of rank $3$.

\begin{lemma}
If $L$, $L'$ are paratopic latin squares then $\pi(L)=\pi(L')$.
\end{lemma}
\begin{proof}
Suppose that $L$ is a division sudoku with respect to the sudoku tri-partition $(X^B,X^S,X^P)=(X^1,X^2,X^3)$. If $(\alpha_1,\alpha_2,\alpha_3)$ is an isotopism from $L$ to $L'$ then $L'$ is a division sudoku with respect to the sudoku tri-partition $(Y^1,Y^2,Y^3)$, where $Y^j_i = \alpha_j(X^j_i)$ for every $1\le i\le m$ and $1\le j\le 3$. Similarly, if $\theta\in\sym{3}$ and $L'=L^\theta$, then $L'$ is a division sudoku with respect to the sudoku tri-partition $(X^{\theta(1)},X^{\theta(2)},X^{\theta(3)})$.
\end{proof}

For instance, $\pi(\DS(9,18))=4$ and we can visualize the situation as follows:

\medskip
\begin{center}
\begin{tabular}{P{35mm}P{35mm}P{35mm}}
\oplabel{}
\rowlabels{1}{2}{3}{4}{5}{6}{7}{8}{9}
\collabels{1}{2}{3}{4}{5}{6}{7}{8}{9}
\smalllabeledsudoku{1&4&7&2&5&8&3&6&9}{8&2&5&9&3&6&7&1&4}{6&9&3&4&7&1&5&8&2}{2&5&8&3&6&9&1&4&7}{9&3&6&7&1&4&8&2&5}{4&7&1&5&8&2&6&9&3}{3&6&9&8&2&5&4&7&1}{7&1&4&6&9&3&2&5&8}{5&8&2&1&4&7&9&3&6}
&
\oplabel{}
\rowlabels{1}{2}{3}{4}{5}{6}{7}{8}{9}
\collabels{1}{4}{7}{2}{5}{8}{3}{6}{9}
\smalllabeledsudoku{1&2&3&4&5&6&7&8&9}{8&9&7&2&3&1&5&6&4}{6&4&5&9&7&8&3&1&2}{2&3&1&5&6&4&8&9&7}{9&7&8&3&1&2&6&4&5}{4&5&6&7&8&9&1&2&3}{3&8&4&6&2&7&9&5&1}{7&6&2&1&9&5&4&3&8}{5&1&9&8&4&3&2&7&6}
&
\oplabel{}
\rowlabels{1}{2}{3}{4}{5}{6}{7}{8}{9}
\collabels{1}{5}{9}{2}{6}{7}{3}{4}{8}
\smalllabeledsudoku{1&5&9&4&8&3&7&2&6}{8&3&4&2&6&7&5&9&1}{6&7&2&9&1&5&3&4&8}{2&6&7&5&9&1&8&3&4}{9&1&5&3&4&8&6&7&2}{4&8&3&7&2&6&1&5&9}{3&2&1&6&5&4&9&8&7}{7&9&8&1&3&2&4&6&5}{5&4&6&8&7&9&2&1&3}
\end{tabular}
\end{center}
\medskip

Here, all three latin squares represent $L=\DS(9,18)$ with the indicated labeling of rows and columns. The first square shows that $L$ is a division sudoku with respect to the tri-partition
\begin{displaymath}
   X^B=\{123\ 456\ 789\},\ X^S=\{123\ 456\ 789\},\ X^P=\{123\ 456\ 789\}
\end{displaymath}
but also with respect to the tri-partition
\begin{displaymath}
    X^B=\{123\ 456\ 789\},\ X^S=\{123\ 456\ 789\},\ X^P=\{159\ 267\ 348\}.
\end{displaymath}
(The first partition $X^P=\{123\ 456\ 789\}$ can be seen on the broken diagonals of the minisquare $M=B_1\cap S_1$ and then in all minisquares, while the second partition $X^P=\{159\ 267\ 348\}$ can be seen on the broken antidiagonals of the minisquare $M$ and then in all minisquares. It is not always the case that for a given $X^B$, $X^S$ there are two suitable $X^P$s obtained in this fashion.)
The second square shows that $L$ is a division sudoku with respect to the tri-partition
\begin{displaymath}
    X^B=\{123\ 456\ 789\},\ X^S=\{147\ 258\ 369\},\ X^P=\{147\ 258\ 369\}.
\end{displaymath}
Finally, the third square shows that $L$ is a division sudoku with respect to the tri-partition
\begin{displaymath}
    X^B=\{123\ 456\ 789\},\ X^S=\{159\ 267\ 348\},\ X^P=\{147\ 258\ 369\}.
\end{displaymath}

\begin{proposition}[\Keyboard]\label{Pr:Tripartitions}
Among the $45$ main ds-classes of division sudokus of rank $3$ (cf. Table \ref{Tb:MainDSClasses}), only $7$ classes consist of latin squares that are division sudokus with respect to more than one sudoku tri-partition, namely:
\begin{displaymath}
\begin{array}{c|cccccccccccccccccc}
    i               &2  &17 &18 &20 &27 &175    &179\\
    \pi(\DS(9,i))    &2  &24 &4  &9  &2  &3      &4
\end{array}
\end{displaymath}
\end{proposition}
\begin{proof}
Given a latin square $L$ of order $9$, all sudoku tri-partitions with respect to which $L$ is a division sudoku can be found quickly as follows. For a given set $B$ of $3$ rows, let $\mathcal S(B)$ be the set of all subsets $S$ of columns such that $B\times S$ is a $3\times 3$ latin subsquare of $L$. There are $\binom{8}{2}\binom{5}{2}=280$ partitions of rows into $3$-element subsets. Given a partition $(X^B_1,X^B_2,X^B_3)$ of rows, we consider all partitions of columns $(X^S_1,X^S_2,X^S_3)$ such that $X^S_i \in \mathcal S(X^B_i)$. Once $(X^B_1,X^B_2,X^B_3)$ and $(X^S_1,X^S_2,X^S_3)$ are given, there are only two candidates for $(X^P_1,X^P_2,X^P_3)$, one corresponding to the broken diagonals of the minisquare $M=X^B_1\times X^S_1$, and another corresponding to the broken antidiagonals of $M$.
\end{proof}

\begin{table}[ht]
\caption{All sudoku tri-partitions with respect to which a given latin square is a division sudoku.}\label{Tb:Tripartitions}
\begin{scriptsize}
\begin{center}
\begin{tabular}{lll}
DS(9,2)&&                                                       DS(9,18)\\
(\,\{123 456 789\},\,\{123 456 789\},\,\{123 456 789\}\,)&&     (\,\{123 456 789\},\,\{123 456 789\},\,\{123 456 789\}\,)\\
(\,\{123 456 789\},\,\{123 456 789\},\,\{159 267 348\}\,)&&     (\,\{123 456 789\},\,\{123 456 789\},\,\{159 267 348\}\,)\\
                                                         &&     (\,\{123 456 789\},\,\{147 258 369\},\,\{147 258 369\}\,)\\
DS(9,17)&&                                                      (\,\{123 456 789\},\,\{159 267 348\},\,\{147 258 369\}\,)\\
(\,\{123 456 789\},\,\{123 456 789\},\,\{123 456 789\}\,)&&     \\
(\,\{123 456 789\},\,\{123 456 789\},\,\{159 267 348\}\,)&&     DS(9,20)\\
(\,\{123 456 789\},\,\{147 258 369\},\,\{159 267 348\}\,)&&     (\,\{123 456 789\},\,\{123 456 789\},\,\{123 456 789\}\,)\\
(\,\{123 456 789\},\,\{147 258 369\},\,\{147 258 369\}\,)&&     (\,\{123 456 789\},\,\{125 389 467\},\,\{127 346 589\}\,)\\
(\,\{123 456 789\},\,\{159 267 348\},\,\{123 456 789\}\,)&&     (\,\{123 456 789\},\,\{127 346 589\},\,\{179 236 458\}\,)\\
(\,\{123 456 789\},\,\{159 267 348\},\,\{147 258 369\}\,)&&     (\,\{123 456 789\},\,\{134 278 569\},\,\{139 256 478\}\,)\\
(\,\{147 258 369\},\,\{123 456 789\},\,\{159 267 348\}\,)&&     (\,\{123 456 789\},\,\{139 256 478\},\,\{125 389 467\}\,)\\
(\,\{147 258 369\},\,\{123 456 789\},\,\{168 249 357\}\,)&&     (\,\{123 456 789\},\,\{145 238 679\},\,\{134 278 569\}\,)\\
(\,\{147 258 369\},\,\{159 267 348\},\,\{168 249 357\}\,)&&     (\,\{123 456 789\},\,\{147 258 369\},\,\{159 267 348\}\,)\\
(\,\{147 258 369\},\,\{159 267 348\},\,\{147 258 369\}\,)&&     (\,\{123 456 789\},\,\{159 267 348\},\,\{147 258 369\}\,)\\
(\,\{147 258 369\},\,\{168 249 357\},\,\{159 267 348\}\,)&&     (\,\{123 456 789\},\,\{179 236 458\},\,\{145 238 679\}\,)\\
(\,\{147 258 369\},\,\{168 249 357\},\,\{147 258 369\}\,)&&     \\
(\,\{159 267 348\},\,\{123 456 789\},\,\{123 456 789\}\,)&&     DS(9,27)\\
(\,\{159 267 348\},\,\{123 456 789\},\,\{168 249 357\}\,)&&     (\,\{123 456 789\},\,\{123 456 789\},\,\{123 456 789\}\,)\\
(\,\{159 267 348\},\,\{147 258 369\},\,\{147 258 369\}\,)&&     (\,\{123 456 789\},\,\{147 258 369\},\,\{147 258 369\}\,)\\
(\,\{159 267 348\},\,\{147 258 369\},\,\{168 249 357\}\,)&&     \\
(\,\{159 267 348\},\,\{168 249 357\},\,\{147 258 369\}\,)&&     DS(9,175)\\
(\,\{159 267 348\},\,\{168 249 357\},\,\{123 456 789\}\,)&&     (\,\{123 456 789\},\,\{123 456 789\},\,\{123 456 789\}\,)\\
(\,\{168 249 357\},\,\{147 258 369\},\,\{159 267 348\}\,)&&     (\,\{147 258 369\},\,\{169 247 358\},\,\{148 259 367\}\,)\\
(\,\{168 249 357\},\,\{147 258 369\},\,\{168 249 357\}\,)&&     (\,\{168 249 357\},\,\{148 259 367\},\,\{169 247 358\}\,)\\
(\,\{168 249 357\},\,\{159 267 348\},\,\{168 249 357\}\,)&&     \\
(\,\{168 249 357\},\,\{159 267 348\},\,\{123 456 789\}\,)&&     DS(9,179)\\
(\,\{168 249 357\},\,\{168 249 357\},\,\{123 456 789\}\,)&&     (\,\{123 456 789\},\,\{123 456 789\},\,\{123 456 789\}\,)\\
(\,\{168 249 357\},\,\{168 249 357\},\,\{159 267 348\}\,)&&     (\,\{149 257 368\},\,\{167 248 359\},\,\{158 269 347\}\,)\\
                                                         &&     (\,\{158 269 347\},\,\{158 269 347\},\,\{167 248 359\}\,)\\
                                                         &&     (\,\{167 248 359\},\,\{149 257 368\},\,\{149 257 368\}\,)\\

\end{tabular}
\end{center}
\end{scriptsize}
\end{table}

Table \ref{Tb:Tripartitions} lists all sudoku tri-partitions with respect to which the latin squares of Proposition \ref{Pr:Tripartitions} are division sudokus.

\subsection{Synchronizing tri-partitions}

A sudoku tri-partition $(X^B,X^S,X^P)$ is \emph{synchronized} if $X^B=X^S=X^P$. Of course, a synchronized sudoku tri-partition can be identified with a sudoku partition.

For a latin square $L$ denote by $\sigma(L)$ the number of synchronized sudoku tri-partitions with respect to which $L$ is a division sudoku. In other words, $\sigma(L)$ is the number of sudoku partitions with respect to which $L$ is a division sudoku.

The value of $\sigma$ is preserved under conjugation but not necessarily under ds-isotopisms. For a ds-isotopism class $\mathcal C$ with respect to the standard sudoku partition, let $\sigma(\mathcal C)=\max\{\sigma(L):L\in \mathcal C\}$. Our task is to determine $\sigma(\mathcal C)$ for every ds-isotopism class $\mathcal C$ and to find at least one $L\in\mathcal C$ such that $\sigma(L)=\sigma(\mathcal C)$.

If $L=\DS(9,i)\in\mathcal C$ for some $1\le i\le 186$, then $\sigma(L)\ge 1$ on account of the standard sudoku partition of $L$. It can certainly happen that $\sigma(L)<\pi(L)$ and, in fact, $\sigma(\mathcal C)<\pi(L)$.

\begin{proposition}[\Keyboard]\label{Pr:MaxSynchro}
For $1\le i\le 186$, let $\mathcal C_i$ be the ds-isotopism class of the division sudoku $\DS(9,i)$. Then $\sigma(\mathcal C_i)=1$ except in the cases $\sigma(\mathcal C_{17})=4$, $\sigma(\mathcal C_{175})=3$ and $\sigma(\mathcal C_{179})=4$.
\end{proposition}
\begin{proof}
Let $L$ be a division sudoku with sudoku tri-partitions $X=(X^B,X^S,X^P)$, $Y=(Y^B,Y^S,Y^P)$, $Z=(Z^B,Z^S,Z^P)$, etc. Suppose without loss of generality that $X^B=X^S=X^P$ is the standard sudoku partition.

Suppose that we wish to find a ds-isotopic copy $L'$ of $L$ for which $X$ is still a (synchronized) sudoku tri-partition and the image of $Y$ is also synchronized. Let $G$ be the subgroup of $\mathrm{Sym}(9)$ that preserves the partition $X^B$. This group acts naturally on the $280$-element set of sudoku partitions, and $G\times G\times G$ acts naturally on the set of sudoku tri-partitions with $X$ a fixed point. Our task is to find $f$, $g$, $h\in G$ such that $(Y^B)^f = (Y^S)^g = (Y^P)^h$, which happens if and only if $Y^B=(Y^S)^{gf^{-1}} = (Y^P)^{hf^{-1}}$. It is therefore possible to synchronize $Y$ while stabilizing $X$ if and only if $Y^B$, $Y^S$, $Y^P$ are in the same orbit of $G$.

If $Y$ cannot be synchronized while stabilizing $X$, we backtrack. Otherwise we find $g$, $h\in G$ such that $Y^B = (Y^S)^g= (Y^P)^h$ and we let $L'$ be the uniquely determined latin square for which $(1,g,h)$ is a ds-isotopism from $L$ to $L'$. Then $L'$ is a division sudoku with respect to the tri-partitions $X'=X$, $Y'=(Y^B,Y^B,Y^B)$, $Z'=(Z^B,(Z^S)^g,(Z^P)^h)$, etc, and both $X'$, $Y'$ are synchronized. Attempting to synchronize $Z'$, we replace the group $G$ with the stabilizer of both $X^B$ and $Y^B$, and we proceed as above.
\end{proof}

The algorithm described in the proof of Proposition \ref{Pr:MaxSynchro} determines $\sigma(\mathcal C_i)$ and $L_i\in \mathcal C_i$ with $\sigma(L_i)=\sigma(\mathcal C_i)$ in a matter of seconds. In particular, it finds the latin squares

\medskip
\begin{center}
\begin{tabular}{P{35mm}P{35mm}P{35mm}}
\settablename{\normalsize{$L_{17}$}}
\namedsmallsudoku{5&2&8&7&4&1&3&9&6}{9&6&3&2&8&5&4&1&7}{1&7&4&6&3&9&8&5&2}{6&3&9&8&5&2&1&7&4}{7&4&1&3&9&6&5&2&8}{2&8&5&4&1&7&9&6&3}{4&1&7&9&6&3&2&8&5}{8&5&2&1&7&4&6&3&9}{3&9&6&5&2&8&7&4&1}
&
\settablename{\normalsize{$L_{175}$}}
\namedsmallsudoku{2&5&8&4&9&1&3&7&6}{9&3&6&2&5&7&4&1&8}{4&7&1&8&3&6&9&5&2}{6&1&9&5&8&2&7&3&4}{7&4&2&3&6&9&5&8&1}{3&8&5&7&1&4&2&6&9}{1&6&7&9&4&3&8&2&5}{8&2&4&1&7&5&6&9&3}{5&9&3&6&2&8&1&4&7}
&
\settablename{\normalsize{$L_{179}$}}
\namedsmallsudoku{7&1&4&9&5&3&8&2&6}{5&8&2&1&7&6&4&9&3}{3&6&9&4&2&8&1&5&7}{8&3&5&2&4&9&7&6&1}{6&9&1&7&3&5&2&8&4}{2&4&7&6&8&1&5&3&9}{9&5&3&8&1&4&6&7&2}{1&7&6&5&9&2&3&4&8}{4&2&8&3&6&7&9&1&5}
\end{tabular}
\end{center}
\medskip

The square $L_{17}\in\mathcal C_{17}$ is a division sudoku with respect to each of the following four sudoku partitions
\begin{displaymath}
\{123\ 456\ 789\},\ \{147\ 258\ 369\},\ \{159\ 267\ 348\},\ \{168\ 249\ 357\}.
\end{displaymath}
The square $L_{175}\in \mathcal C_{175}$ is a division sudoku with respect to each of the following three sudoku partitions
\begin{displaymath}
\{123\ 456\ 789\},\ \{147\ 258\ 369\},\ \{168\ 249\ 357\}.
\end{displaymath}
The square $L_{179}\in\mathcal C_{179}$ is a division sudoku with respect to each of the following four sudoku partitions
\begin{displaymath}
\{123\ 456\ 789\},\ \{149\ 257\ 368\},\ \{158\ 269\ 347\},\ \{167\ 248\ 359\}.
\end{displaymath}

\begin{remark}
The division sudoku $L_0$ from the introduction (which is ds-isotopic to $\DS(9,179)$) is a division sudoku with respect to exactly the same sudoku partitions as $L_{17}$. Upon considering the ds-isotopism $((7,8,9),(7,8,9),(7,8,9))$, we can transform $L_{179}$ into a square $L'_{179}$ which also has exactly the same sudoku partitions as $L_{17}$.
\end{remark}

For $m\ge 3$, let
\begin{displaymath}
    \sigma(m) = \max\{\sigma(L):\text{$L$ is a division sudoku of rank $m$}\}
\end{displaymath}
be the maximum number of division sudoku partitions that a division sudoku of rank $m$ can possess.

\begin{corollary}
$\sigma(3)=4$.
\end{corollary}

\section{Division sudokus constructed from fields and nearfields}\label{Sc:Nearfields}

In this section we construct division sudokus of prime power rank that are rich in division sudoku partitions.

\subsection{Division sudokus constructed from fields}

For a prime power $q$, let $\F_q$ be the field of order $q$. Given $c\in\F_{q^2}\setminus\{0,1\}$, define a multiplication $*$ on $\F_{q^2}$ by
\begin{equation}\label{Eq:Stein}
    x*y = x+(y-x)c.
\end{equation}
It is then easy to see that $(\F_{q^2},*)$ is a quasigroup with left and right divisions given by
\begin{displaymath}
    x\ldiv y = x + (y-x)c^{-1},\quad x\rdiv y = x + (y-x)c(c-1)^{-1} = y + (x-y)(1-c)^{-1}.
\end{displaymath}
respectively.

\begin{definition}
Let $q$ be a prime power and $c\in \F_{q^2}\setminus \{0,1\}$. Then the quasigroup $(\F_{q^2},*)$ with multiplication \eqref{Eq:Stein} will be denoted by $\DS(\F_{q^2},c)$.
\end{definition}

We will show that $\DS(\F_{q^2},c)$ can be equipped with a number of sudoku partitions with respect to which it is a division sudoku.

\begin{lemma}\label{Lm:w1}
Let $V$ be a field seen as a vector space over some subfield $F$. Let $W$ be an $F$-subspace of $V$ and $c\in V\setminus\{0,1\}$. Then the following conditions are equivalent: $W\cap Wc = 0$, $W\cap W(c-1) = 0$ and $Wc \cap W(c-1) = 0$.
\end{lemma}
\begin{proof}
Let $w$, $w'$ be elements of $W$. Note that $w' = wc$ if and only if $w'-w = w(c-1)$, and $w'(c-1) = wc$ if and only if $w' = (w'-w)c$.
\end{proof}

\begin{proposition}\label{Pr:w2}
Let $r$ be a prime power, $q=r^s$, $V=\F_{q^2}$ a vector space over $\F_r$ and $W$ an $\F_r$-subspace of $V$ of dimension $s$. Let $c\in V$ be such that $W\cap Wc=0$. For $x$, $y \in V$ write $x\sim y$ if and only if $x-y \in W$. Then $\DS(\F_{q^2},c)$ is a division sudoku of rank $q$ with respect to the partition induced by the equivalence relation $\sim$.
\end{proposition}
\begin{proof}
Since $|V|=q^2$ and every coset of $W$ has cardinality $q=r^s$, the partition induced by $\sim$ is a sudoku partition with blocks of size $q$. We will use Proposition \ref{Pr:DSPerms} to check that $\DS(\F_{q^2},c)$ is a division sudoku with respect to $\sim$. Suppose that $u\sim v$, i.e., $u-v\in W$.

If $x/u \sim x/v$ then $(x+(u-x)c(c-1)^{-1})-(x+(v-x)c(c-1)^{-1}) = w\in W$, so $(u-v)c=w(c-1)\in Wc\cap W(c-1)=0$ by Lemma \ref{Lm:w1}. If $x*u \sim x*v$ then $(x+(u-x)c) - (x+(v-x)c) = w'\in W$, so $(u-v)c=w'\in W\cap Wc=0$. If $u*x\sim v*x$ then $(u+(x-u)c) - (v+(x-v)c) = w''\in W$, so $(u-v)(1-c) = w''\in W\cap W(1-c)=0$ by Lemma \ref{Lm:w1}. In each case we conclude that $u=v$.
\end{proof}

\begin{theorem}\label{Th:DSF}
Let $q$ be a prime power and $c\in\F_{q^2}\setminus\F_q$. Then $L=\DS(\F_{q^2},c)$ is a division sudoku of rank $q$ satisfying $\sigma(L)\ge q+1$. In particular, $\sigma(q)\ge q+1$.
\end{theorem}
\begin{proof}
Consider $V=\F_{q^2}$ as a vector space over $\F_q$ and let $W$ be a one-dimensional $\F_q$-subspace of $V$. Note that there are $(q^2-1)/(q-1)=q+1$ such subspaces. We claim that $W\cap Wc=0$. Indeed, we have $W=w\F_q$ for some $0\ne w\in V$ and if $W\cap Wc\ne 0$ then $w\lambda = w\mu c$ for some $\lambda$, $\mu\in\F_q$, which implies $c = \mu^{-1}\lambda\in\F_q$, a contradiction. The rest follows from Proposition \ref{Pr:w2} with $r=q$ (and thus with $s=1$).
\end{proof}

The division sudokus $\DS(\F_{q^2},c)$ from Theorem \ref{Th:DSF} are isotopic to the elementary abelian group $C_q\times C_q$. For $c\in \F_{3^2}\setminus \F_3$, the division sudoku $\DS(\F_{3^2},c)$ is ds-isotopic to the division sudoku $L_{17}$ from the proof of Proposition \ref{Pr:MaxSynchro} and also to the example in \cite[Figure 5]{BaileyEtAl}.

We proceed to find more suitable subspaces and thus more sudoku partitions in quartic field extensions.

\begin{proposition}\label{Pr:w3}
Let $r$ be a prime power, $q=r^2$ and $\F_r\le\F_q\le\F_{q^2}$. Let $c\in \F_{q^2}\setminus\F_q$. Then there are exactly $q^2+q$ two-dimensional $\F_r$-subspaces $W$ of $\F_{q^2}$ such that $W\cap Wc=0$.
\end{proposition}
\begin{proof}
Let $Y$ be the set of all two-dimensional $\F_r$-subspaces $W$ of $V=\F_{q^2}$. Then
\begin{displaymath}
    |Y| = \frac{(r^4-1)(r^4-r)}{(r^2-1)(r^2-r)} = r^4+r^3+2r^2+r+1.
\end{displaymath}
Let $Y'=\{v\F_q:v\in V\}$ be the set of all one-dimensional $\F_q$-subspaces $W'$ of $V$. We have
\begin{displaymath}
    |Y'|= \frac{r^4-1}{r^2-1} = r^2+1.
\end{displaymath}
The cyclic group $\langle c\rangle$ acts on $Y$ and $Y'$ by multiplication. Note that $W'\cap W'c=0$ for every $W'\in Y'$, else $v\F_q\cap v\F_q c\ne 0$ implies $c\in \F_q$, a contradiction.

Consider a ``bad'' subspace $W\in Y$ such that $W\cap Wc\ne 0$, say $W\cap Wc=W_1c$ for some one-dimensional $\F_r$-subspace $W_1$ of $W$. We have $W_1c\ne W_1$ (else $c\in \F_r$) and hence $W=W_1+W_1c$. Since $\bigcup Y' = V$, there is $W'\in Y'$ such that $W_1\le W'$. Summarizing, we have $W=W_1+W_1c$, where $W_1$ is a one-dimensional $\F_r$-subspace of some $W'\in Y'$.

Conversely, let $W=W_1+W_1c$, where $W_1$ is a one-dimensional $\F_r$-subspace of some $W'\in Y'$. Then $W\in Y$ is bad since $W\cap Wc = (W_1+W_1c)\cap(W_1c+W_1c^2)\ne 0$.

Let us count the bad subspaces. We have $|Y'|=r^2+1$ choices of $W'$ and then $(r^2-1)/(r-1)=r+1$ choices of $W_1$ in $W'$. Note that $W_1\ne W_1c^2$ (else $c^2\in\F_r$ and $c\in\F_q$). Hence we did not count any bad subspace more than once. Altogether, there are $|Y|-(r^2+1)(r+1) = r^4+r^2$ subspaces $W\in Y$ such that $W\cap Wc=0$.
\end{proof}

Combining Propositions \ref{Pr:w2} and \ref{Pr:w3}, we obtain:

\begin{theorem}\label{Th:last}
Let $q=p^{2s}$ be a prime power and $c\in\F_{q^2}\setminus\F_q$. Then $L=\DS(\F_{q^2},c)$ is a division sudoku of rank $q$ satisfying $\sigma(L)\ge q^2+q = p^{4s}+p^{2s}$. In particular, $\sigma(p^{2s})\ge p^{4s}+p^{2s}$.
\end{theorem}

Applying Theorem \ref{Th:last} with $q=4$, we obtain a division sudoku $L$ of rank $4$ with $\sigma(L)\ge 20$.

\subsection{Division sudokus constructed from nearfields}

A (\emph{left}) \emph{nearfield} is an algebra $(F,+,\circ,0,1)$ such that $(F,+,0)$ is an abelian group, $(F,\circ,1)$ is a monoid, every non-zero element $x$ has an inverse $\overline{x}$ such that $x\circ \overline{x}=\overline{x}\circ x = 1$, and $x\circ(y+z) = (x\circ y)+(x\circ z)$ holds for every $x$, $y$, $z\in F$. All finite nearfields were constructed by Dickson and their classification was completed by Zassenhaus \cite{Zassenhaus}.

We will use the following lemma without reference:

\begin{lemma}
Let $F$ be a nearfield. Then:
\begin{enumerate}
\item[(i)] $0\circ x = x\circ 0 = 0$ for every $x\in F$,
\item[(ii)] $(-1)\circ x = -x = x\circ (-1)$ for every $x\in F$,
\item[(iii)] $(-x)\circ y = x\circ (-y) = -(x\circ y)$ for every $x$, $y\in F$.
\end{enumerate}
\end{lemma}
\begin{proof}
For (i) and (ii) see \cite{Zassenhaus}. Then $-(x\circ y) = (-1)\circ x\circ y$ is equal to $(-x)\circ y$ and also to $x\circ (-1)\circ y = x\circ (-y)$.
\end{proof}

Stein proved in \cite[Theorem 2.5]{Stein} that for every \emph{finite} nearfield $(F,+,\circ,0,1)$ and every $c\in F\setminus \{0,1\}$, the groupoid $(F,*)$ defined by
\begin{equation}\label{Eq:SteinCirc}
    x*y = x+(y-x)\circ c
\end{equation}
is a quasigroup. (We remark that the conclusion remains true for infinite nearfields as well. The left division is given by $x\ldiv y = x+(y-x)\circ \overline{c}$, and the right division by $x\rdiv y = x+(y-x)\circ(1-\overline{1-c})$, cf. \cite{DV}.)

\begin{definition}
Let $q$ be an odd prime power. Define $\D_{q^2} = (\F_{q^2},+,\circ,0,1)$ by modifying the multiplication in the field $\F_{q^2}$ as follows:
\begin{displaymath}
    x\circ y = \left\{\begin{array}{ll}
        xy,&\text{ if $x$ is a square in $\F_{q^2}$},\\
        xy^q,&\text{ otherwise},
    \end{array}\right.
\end{displaymath}
where the power $y^q$ is taken in $\F_{q^2}$. Then $\D_{q^2}$ is a nearfield, so called \emph{quadratic nearfield}.
\end{definition}

It is not difficult to see that $\lambda x = \lambda\circ x = x\circ\lambda$ for every $\lambda\in\F_q$ and $x\in\D_{q^2}$.

\begin{definition}
Let $q$ be an odd prime power, $\D_{q^2}=(\F_{q^2},+,\circ,0,1)$ the quadratic nearfield, and $c\in\F_{q^2}\setminus \{0,1\}$. Then the quasigroup $(\D_{q^2},*)$ with Stein's multiplication \eqref{Eq:SteinCirc} will be denoted by $\DS(\D_{q^2},c)$.
\end{definition}

\begin{lemma}\label{Lm:ZeroSum}
Let $a$, $b\in \F_{q^2}$, $c\in\F_{q^2}\setminus\F_q$ and $0\ne w\in \F_{q^2}$. Then the following conditions hold in $\D_{q^2}$:
\begin{enumerate}
\item[(i)] If $a+b \in \F_q$ and $a\circ c + b\circ c\in \F_q$ then $a+b=0$.
\item[(ii)] If $a+b \in w\F_q$ and $a\circ c + b\circ c\in w\F_q$ then $a+b = 0$.
\end{enumerate}
\end{lemma}
\begin{proof}
(i) If both $a$, $b$ are squares in $\F_{q^2}$ then $(a+b)c = ac+bc = a\circ c + b\circ c\in\F_q$, which implies $a+b=0$, otherwise we deduce from $0\ne a+b\in\F_q$ and $(a+b)c\in\F_q$ than $c\in\F_q$, a contradiction.

If both $a$, $b$ are nonsquares then $a\circ c + b\circ c = ac^q+bc^q = (a+b)c^q\in\F_q$. But then $\F_q$ contains $((a+b)c^q)^q = (a+b)^qc^{q^2} =(a+b)c$, too, where we have used $a+b\in\F_q$. We conclude $a+b=0$ as above.

Now suppose that $a$ is a square and $b$ is a nonsquare, let $\lambda=a+b$ and $d =  a\circ c + b\circ c = ac + bc^q\in \F_q$. Then $ac+(\lambda-a)c^q = ac+bc^q = d = d^q = a^qc^q + (\lambda-a^q)c$ and hence $0 =(a+a^q-\lambda)(c-c^q)=(a^q-b)(c-c^q)$. We have $c\ne c^q$ since $c\not\in\F^q$ and therefore $a^q=b$, a contradiction with the assumption that $a$ is a square while $b$ is a nonsquare.

(ii) Let $v$ be the multiplicative inverse of $w$ in $\D_{q^2}$, i.e., $v=\overline{w}$. Then the assumptions of (ii) can be rewritten as $v \circ a + v \circ b \in \mathbb F_q$ and $(v \circ a) \circ c + (v \circ b) \circ c \in \mathbb F_q$, using associativity of $\circ$. By (i), $0 = v \circ a + v \circ b  = v\circ(a+b)$, which implies $a+b=0$.
\end{proof}

\begin{proposition}\label{Pr:InducesDSDickson}
Let $q$ be an odd prime power, $V=\F_{q^2}$ a vector space over $\F_q$, $W$ a one-dimensional $\F_q$-subspace of $V$, and $c\in\F_{q^2}\setminus\F_q$. For $x$, $y \in V$ write $x\sim y$ if and only if $x-y \in W$. Then $\DS(\D_{q^2},c)$ is a division sudoku of rank $q$ with respect to the partition induced by the equivalence relation $\sim$.
\end{proposition}
\begin{proof}
Let $W=w\F_q$ for some $0\ne w\in V$. We will use \eqref{Eq:DSQuasi1}, \eqref{Eq:DSPerms2} and \eqref{Eq:DSPerms3} to show that $\DS(\D_{q^2},c)$ is a division sudoku with respect to $\sim$.

Suppose that $x\sim u$, $y\sim v$ and $x*y=u*v$. Then $x+(y-x)\circ c=u+(v-u)\circ c$ and therefore $(y-x)\circ c + (u-v)\circ c = u - x \in w\F_q$. We also have $(y-x)+(u-v) = (u-x) + (y-v) \in w\F_q$. By Lemma \ref{Lm:ZeroSum}(ii), $(u-x)+(y-v)=0$, or, equivalently, $y-x=v-u$. Combining this with $x+(y-x)\circ c=u+(v-u)\circ c$ yields $x=u$ and then $y=v$.

Now suppose that $u\sim v$ and $x*u\sim x*v$. Then $ (x+(u-x)\circ c) - (x+(v-x)\circ c) = (u-x)\circ c + (x-v)\circ c$ is in $w\F_q$ and so is $(u-x)+(x-v) =u-v$. By Lemma \ref{Lm:ZeroSum}(ii), $u=v$.

Finally, suppose that $u\sim v$ and $u*x\sim v*x$. Then $(u+(x-u)\circ c)-(v+(x-v)\circ c)$ is in $w\F_q$ and so is $u-v$. Combining, we deduce that $(x-u)\circ c+(v-x)\circ c\in w\F_q$. We also have $(x-u)+(v-x) = v-u\in w\F_q$, and Lemma \ref{Lm:ZeroSum}(ii) yields $u=v$.
\end{proof}

\begin{theorem}
Let $q$ be an odd prime power and $c\in\F_{q^2}\setminus\F_q$. Then $L=\DS(\D_{q^2},c)$ is a division sudoku satisfying $\sigma(L)\ge q+1$.
\end{theorem}
\begin{proof}
Apply Proposition \ref{Pr:InducesDSDickson} to each of the $q+1$ one-dimensional $\F_q$-subspaces of $\F_{q^2}$.
\end{proof}

For an odd prime power $q$ and $c_1$, $c_2\in \F_{q^2}\setminus \F_q$, the division sudokus $\DS(\F_{q^2},c_1)$, $\DS(\D_{q^2},c_2)$ are not isotopic since the former is isotopic to a group while the latter is not.
For every $c\in\F_{3^2}\setminus \F_3$, the division sudoku $\DS(\D_{3^2},c)$ is ds-isotopic to the division sudoku $L_{179}$ from the proof of Proposition \ref{Pr:MaxSynchro} and also to the example $L_0$ from the introduction.

\section{Open problems}

\begin{problem}
Determine the number of standard division sudokus of rank $4$; absolutely, up to ds-isotopism and up to ds-paratopism.
\end{problem}

\begin{problem}
For $m>3$, find lower and upper bounds for the number of standard division sudokus of rank $m$.
\end{problem}

Recall that $\sigma(L)$ is the number of sudoku partitions with respect to which $L$ is a division sudoku and
\begin{displaymath}
    \sigma(m) = \max\{\sigma(L):L\text{ is a division sudoku of rank }m\}.
\end{displaymath}
Our results imply that $\sigma(3)=4$, $\sigma(q)\ge q+1$ for any prime power $q$ and $\sigma(p^{2s})\ge p^{4s}+p^{2s}$ for any prime $p$.

\begin{problem}
Investigate $\sigma(m)$ for $m>3$.
\end{problem}

For a prime power $q$ and $c\in\F_{q^2}\setminus\F_q$ consider the division sudoku $L=\DS(\F_{q^2},c)$. By Theorem \ref{Th:DSF}, $\sigma(L)\ge q+1$. For a fixed one-dimensional $\F_q$-subspace $W$ of $\F_{q^2}$, the blocks of the corresponding sudoku partition are all the lines in $\F_{q^2}$ parallel to $W$. Taken together, the blocks of the $q+1$ sudoku partitions for $L$ form an affine plane. Consequently, two blocks from distinct sudoku partitions of $L$ intersect in precisely one point.

Let $C$ be a collection of sudoku partitions on a latin square $L$ of rank $m$ (not necessarily a prime power) such that $L$ is a division sudoku with respect to each element of $C$. Then $C$ is said to be \emph{affine} if any two blocks from two distinct sudoku partitions of $C$ intersect in precisely one point.

\begin{problem}
For a division sudoku $L$ let $\tau(L)$ be the cardinality of a largest affine collection of sudoku partitions on $L$. Let
\begin{displaymath}
    \tau(m)=\max\{\tau(L):L\text{ is a division sudoku of rank }m\}.
\end{displaymath}
What is $\tau(m)$? Is $\tau(m)=N(m)+2$, where $N(m)$ is the maximal number of mutually orthogonal latin squares of order $m$?
\end{problem}

Let $L$ be a standard division sudoku. A subset $S$ of $L$ is a \emph{critical set} if $S$ determines $L$ as a standard division sudoku and no proper subset of $S$ does.

\begin{problem}
For $m\ge 3$, determine the cardinality of a smallest critical set among all standard division sudokus of rank $m$.
\end{problem}

\appendix

\section*{Appendix: Standard division sudokus of rank $3$ up to ds-isotopism}

\bigskip

\begin{center}
\begin{tabular}{P{28mm}P{28mm}P{28mm}P{28mm}}
\settablename{DS(9,1)}\namedsmallsudoku{1&4&7&2&5&8&3&6&9}{8&2&5&9&3&6&7&1&4}{6&9&3&4&7&1&5&8&2}{2&5&8&1&6&9&4&3&7}{9&3&6&7&2&4&8&5&1}{4&7&1&5&8&3&2&9&6}{3&8&4&6&1&7&9&2&5}{5&1&9&8&4&2&6&7&3}{7&6&2&3&9&5&1&4&8}
&\settablename{DS(9,2)}\namedsmallsudoku{1&4&7&2&5&8&3&6&9}{8&2&5&9&3&6&7&1&4}{6&9&3&4&7&1&5&8&2}{2&5&8&1&4&7&6&9&3}{9&3&6&8&2&5&1&4&7}{4&7&1&6&9&3&8&2&5}{3&6&9&5&8&2&4&7&1}{7&1&4&3&6&9&2&5&8}{5&8&2&7&1&4&9&3&6}
&\settablename{DS(9,3)}\namedsmallsudoku{1&4&7&2&5&8&3&6&9}{8&2&5&9&3&6&7&1&4}{6&9&3&4&7&1&5&8&2}{2&5&8&1&4&7&6&9&3}{9&3&6&8&2&5&1&4&7}{4&7&1&6&9&3&8&2&5}{3&8&4&5&1&9&2&7&6}{5&1&9&7&6&2&4&3&8}{7&6&2&3&8&4&9&5&1}
&\settablename{DS(9,4)}\namedsmallsudoku{1&4&7&2&5&8&3&6&9}{8&2&5&9&3&6&7&1&4}{6&9&3&4&7&1&5&8&2}{2&5&8&1&4&9&6&3&7}{9&3&6&7&2&5&8&4&1}{4&7&1&6&8&3&2&9&5}{3&6&9&5&1&7&4&2&8}{7&1&4&8&6&2&9&5&3}{5&8&2&3&9&4&1&7&6}
\\
\\
\settablename{DS(9,5)}\namedsmallsudoku{1&4&7&2&5&8&3&6&9}{8&2&5&9&3&6&7&1&4}{6&9&3&4&7&1&5&8&2}{2&5&8&1&4&9&6&3&7}{9&3&6&7&2&5&8&4&1}{4&7&1&6&8&3&2&9&5}{3&8&4&5&1&7&9&2&6}{5&1&9&8&6&2&4&7&3}{7&6&2&3&9&4&1&5&8}
&\settablename{DS(9,6)}\namedsmallsudoku{1&4&7&2&5&8&3&6&9}{8&2&5&9&3&6&7&1&4}{6&9&3&4&7&1&5&8&2}{2&5&8&1&4&9&6&3&7}{9&3&6&7&2&5&8&4&1}{4&7&1&6&8&3&2&9&5}{3&8&4&5&9&2&1&7&6}{5&1&9&3&6&7&4&2&8}{7&6&2&8&1&4&9&5&3}
&\settablename{DS(9,7)}\namedsmallsudoku{1&4&7&2&5&8&3&6&9}{8&2&5&9&3&6&7&1&4}{6&9&3&4&7&1&5&8&2}{2&5&8&1&4&9&6&7&3}{9&3&6&7&2&5&1&4&8}{4&7&1&6&8&3&9&2&5}{3&6&9&8&1&4&2&5&7}{7&1&4&5&9&2&8&3&6}{5&8&2&3&6&7&4&9&1}
&\settablename{DS(9,8)}\namedsmallsudoku{1&4&7&2&5&8&3&6&9}{8&2&5&9&3&6&7&1&4}{6&9&3&4&7&1&5&8&2}{2&5&8&1&4&9&6&7&3}{9&3&6&7&2&5&1&4&8}{4&7&1&6&8&3&9&2&5}{3&8&4&5&1&7&2&9&6}{5&1&9&8&6&2&4&3&7}{7&6&2&3&9&4&8&5&1}
\\
\\
\settablename{DS(9,9)}\namedsmallsudoku{1&4&7&2&5&8&3&6&9}{8&2&5&9&3&6&7&1&4}{6&9&3&4&7&1&5&8&2}{2&5&8&1&6&9&4&3&7}{9&3&6&7&2&4&8&5&1}{4&7&1&5&8&3&2&9&6}{3&6&9&8&4&2&1&7&5}{7&1&4&3&9&5&6&2&8}{5&8&2&6&1&7&9&4&3}
&\settablename{DS(9,10)}\namedsmallsudoku{1&4&7&2&5&8&3&6&9}{8&2&5&9&3&6&7&1&4}{6&9&3&4&7&1&5&8&2}{2&5&8&1&6&9&4&3&7}{9&3&6&7&2&4&8&5&1}{4&7&1&5&8&3&2&9&6}{3&8&4&6&9&2&1&7&5}{5&1&9&3&4&7&6&2&8}{7&6&2&8&1&5&9&4&3}
&\settablename{DS(9,11)}\namedsmallsudoku{1&4&7&2&5&8&3&6&9}{8&2&5&9&3&6&7&1&4}{6&9&3&4&7&1&5&8&2}{2&5&8&1&6&9&4&7&3}{9&3&6&7&2&4&1&5&8}{4&7&1&5&8&3&9&2&6}{3&8&4&6&1&7&2&9&5}{5&1&9&8&4&2&6&3&7}{7&6&2&3&9&5&8&4&1}
&\settablename{DS(9,12)}\namedsmallsudoku{1&4&7&2&5&8&3&6&9}{8&2&5&9&3&6&7&1&4}{6&9&3&4&7&1&5&8&2}{2&5&8&1&9&4&6&3&7}{9&3&6&5&2&7&8&4&1}{4&7&1&8&6&3&2&9&5}{3&6&9&7&1&5&4&2&8}{7&1&4&6&8&2&9&5&3}{5&8&2&3&4&9&1&7&6}
\\
\\
\settablename{DS(9,13)}\namedsmallsudoku{1&4&7&2&5&8&3&6&9}{8&2&5&9&3&6&7&1&4}{6&9&3&4&7&1&5&8&2}{2&5&8&1&9&4&6&3&7}{9&3&6&5&2&7&8&4&1}{4&7&1&8&6&3&2&9&5}{3&8&4&7&1&5&9&2&6}{5&1&9&6&8&2&4&7&3}{7&6&2&3&4&9&1&5&8}
&\settablename{DS(9,14)}\namedsmallsudoku{1&4&7&2&5&8&3&6&9}{8&2&5&9&3&6&7&1&4}{6&9&3&4&7&1&5&8&2}{2&5&8&1&9&4&6&7&3}{9&3&6&5&2&7&1&4&8}{4&7&1&8&6&3&9&2&5}{3&6&9&7&4&2&8&5&1}{7&1&4&3&8&5&2&9&6}{5&8&2&6&1&9&4&3&7}
&\settablename{DS(9,15)}\namedsmallsudoku{1&4&7&2&5&8&3&6&9}{8&2&5&9&3&6&7&1&4}{6&9&3&4&7&1&5&8&2}{2&5&8&1&9&4&6&7&3}{9&3&6&5&2&7&1&4&8}{4&7&1&8&6&3&9&2&5}{3&8&4&6&1&9&2&5&7}{5&1&9&7&4&2&8&3&6}{7&6&2&3&8&5&4&9&1}
&\settablename{DS(9,16)}\namedsmallsudoku{1&4&7&2&5&8&3&6&9}{8&2&5&9&3&6&7&1&4}{6&9&3&4&7&1&5&8&2}{2&5&8&1&9&4&6&7&3}{9&3&6&5&2&7&1&4&8}{4&7&1&8&6&3&9&2&5}{3&8&4&7&1&5&2&9&6}{5&1&9&6&8&2&4&3&7}{7&6&2&3&4&9&8&5&1}
\\
\\
\settablename{DS(9,17)}\namedsmallsudoku{1&4&7&2&5&8&3&6&9}{8&2&5&9&3&6&7&1&4}{6&9&3&4&7&1&5&8&2}{2&5&8&3&6&9&1&4&7}{9&3&6&7&1&4&8&2&5}{4&7&1&5&8&2&6&9&3}{3&6&9&1&4&7&2&5&8}{7&1&4&8&2&5&9&3&6}{5&8&2&6&9&3&4&7&1}
&\settablename{DS(9,18)}\namedsmallsudoku{1&4&7&2&5&8&3&6&9}{8&2&5&9&3&6&7&1&4}{6&9&3&4&7&1&5&8&2}{2&5&8&3&6&9&1&4&7}{9&3&6&7&1&4&8&2&5}{4&7&1&5&8&2&6&9&3}{3&6&9&8&2&5&4&7&1}{7&1&4&6&9&3&2&5&8}{5&8&2&1&4&7&9&3&6}
&\settablename{DS(9,19)}\namedsmallsudoku{1&4&7&2&5&8&3&6&9}{8&2&5&9&3&6&7&1&4}{6&9&3&4&7&1&5&8&2}{2&5&8&3&6&9&1&4&7}{9&3&6&7&1&4&8&2&5}{4&7&1&5&8&2&6&9&3}{3&8&4&1&9&5&2&7&6}{5&1&9&6&2&7&4&3&8}{7&6&2&8&4&3&9&5&1}
&\settablename{DS(9,20)}\namedsmallsudoku{1&4&7&2&5&8&3&6&9}{8&2&5&9&3&6&7&1&4}{6&9&3&4&7&1&5&8&2}{2&5&8&3&6&9&1&4&7}{9&3&6&7&1&4&8&2&5}{4&7&1&5&8&2&6&9&3}{3&8&4&6&2&7&9&5&1}{5&1&9&8&4&3&2&7&6}{7&6&2&1&9&5&4&3&8}
\\
\\
\settablename{DS(9,21)}\namedsmallsudoku{1&4&7&2&5&8&3&6&9}{8&2&5&9&3&6&7&1&4}{6&9&3&4&7&1&5&8&2}{2&5&8&3&6&9&4&7&1}{9&3&6&7&1&4&2&5&8}{4&7&1&5&8&2&9&3&6}{3&6&9&1&4&7&8&2&5}{7&1&4&8&2&5&6&9&3}{5&8&2&6&9&3&1&4&7}
&\settablename{DS(9,22)}\namedsmallsudoku{1&4&7&2&5&8&3&6&9}{8&2&5&9&3&6&7&1&4}{6&9&3&4&7&1&5&8&2}{2&5&8&3&6&9&4&7&1}{9&3&6&7&1&4&2&5&8}{4&7&1&5&8&2&9&3&6}{3&8&4&1&9&5&6&2&7}{5&1&9&6&2&7&8&4&3}{7&6&2&8&4&3&1&9&5}
&\settablename{DS(9,23)}\namedsmallsudoku{1&4&7&2&5&8&3&6&9}{8&2&5&9&3&6&7&1&4}{6&9&3&4&7&1&5&8&2}{2&5&8&3&6&9&4&7&1}{9&3&6&7&1&4&2&5&8}{4&7&1&5&8&2&9&3&6}{3&8&4&6&2&7&1&9&5}{5&1&9&8&4&3&6&2&7}{7&6&2&1&9&5&8&4&3}
&\settablename{DS(9,24)}\namedsmallsudoku{1&4&7&2&5&8&3&6&9}{8&2&5&9&3&6&7&1&4}{6&9&3&4&7&1&5&8&2}{2&5&8&3&9&4&1&7&6}{9&3&6&5&1&7&4&2&8}{4&7&1&8&6&2&9&5&3}{3&8&4&7&2&5&6&9&1}{5&1&9&6&8&3&2&4&7}{7&6&2&1&4&9&8&3&5}
\end{tabular}
\end{center}

\pagebreak

\begin{center}
\begin{tabular}{P{28mm}P{28mm}P{28mm}P{28mm}}
\settablename{DS(9,25)}\namedsmallsudoku{1&4&7&2&5&8&3&6&9}{8&2&5&9&3&6&7&1&4}{6&9&3&4&7&1&5&8&2}{2&5&8&3&9&4&6&7&1}{9&3&6&5&1&7&2&4&8}{4&7&1&8&6&2&9&3&5}{3&8&4&6&2&9&1&5&7}{5&1&9&7&4&3&8&2&6}{7&6&2&1&8&5&4&9&3}
&\settablename{DS(9,26)}\namedsmallsudoku{1&4&7&2&5&8&3&6&9}{8&2&5&9&3&6&7&1&4}{6&9&3&4&7&1&5&8&2}{2&5&8&3&9&4&6&7&1}{9&3&6&5&1&7&2&4&8}{4&7&1&8&6&2&9&3&5}{3&8&4&7&2&5&1&9&6}{5&1&9&6&8&3&4&2&7}{7&6&2&1&4&9&8&5&3}
&\settablename{DS(9,27)}\namedsmallsudoku{1&4&7&2&5&8&3&6&9}{8&2&5&9&3&6&7&1&4}{6&9&3&4&7&1&5&8&2}{2&5&8&7&1&4&6&9&3}{9&3&6&5&8&2&1&4&7}{4&7&1&3&6&9&8&2&5}{3&8&4&1&9&5&2&7&6}{5&1&9&6&2&7&4&3&8}{7&6&2&8&4&3&9&5&1}
&\settablename{DS(9,28)}\namedsmallsudoku{1&4&7&2&5&8&3&6&9}{8&2&5&9&3&6&7&1&4}{6&9&3&4&7&1&5&8&2}{2&5&9&1&6&7&4&3&8}{7&3&6&8&2&4&9&5&1}{4&8&1&5&9&3&2&7&6}{3&6&8&7&4&2&1&9&5}{9&1&4&3&8&5&6&2&7}{5&7&2&6&1&9&8&4&3}
\\
\\
\settablename{DS(9,29)}\namedsmallsudoku{1&4&7&2&5&8&3&6&9}{8&2&5&9&3&6&7&1&4}{6&9&3&4&7&1&5&8&2}{2&5&9&1&6&7&4&3&8}{7&3&6&8&2&4&9&5&1}{4&8&1&5&9&3&2&7&6}{3&7&4&6&8&2&1&9&5}{5&1&8&3&4&9&6&2&7}{9&6&2&7&1&5&8&4&3}
&\settablename{DS(9,30)}\namedsmallsudoku{1&4&7&2&5&8&3&6&9}{8&2&5&9&3&6&7&1&4}{6&9&3&4&7&1&5&8&2}{2&5&9&1&8&4&6&3&7}{7&3&6&5&2&9&8&4&1}{4&8&1&7&6&3&2&9&5}{3&7&4&6&9&2&1&5&8}{5&1&8&3&4&7&9&2&6}{9&6&2&8&1&5&4&7&3}
&\settablename{DS(9,31)}\namedsmallsudoku{1&4&7&2&5&8&3&6&9}{8&2&5&9&3&6&7&1&4}{6&9&3&4&7&1&5&8&2}{2&5&9&1&8&4&6&3&7}{7&3&6&5&2&9&8&4&1}{4&8&1&7&6&3&2&9&5}{3&7&4&8&1&5&9&2&6}{5&1&8&6&9&2&4&7&3}{9&6&2&3&4&7&1&5&8}
&\settablename{DS(9,32)}\namedsmallsudoku{1&4&7&2&5&8&3&6&9}{8&2&5&9&3&6&7&1&4}{6&9&3&4&7&1&5&8&2}{2&5&9&1&8&4&6&7&3}{7&3&6&5&2&9&1&4&8}{4&8&1&7&6&3&9&2&5}{3&7&4&8&1&5&2&9&6}{5&1&8&6&9&2&4&3&7}{9&6&2&3&4&7&8&5&1}
\\
\\
\settablename{DS(9,33)}\namedsmallsudoku{1&4&7&2&5&8&3&6&9}{8&2&5&9&3&6&7&1&4}{6&9&3&4&7&1&5&8&2}{2&5&9&3&6&7&1&4&8}{7&3&6&8&1&4&9&2&5}{4&8&1&5&9&2&6&7&3}{3&6&8&1&4&9&2&5&7}{9&1&4&7&2&5&8&3&6}{5&7&2&6&8&3&4&9&1}
&\settablename{DS(9,34)}\namedsmallsudoku{1&4&7&2&5&8&3&6&9}{8&2&5&9&3&6&7&1&4}{6&9&3&4&7&1&5&8&2}{2&5&9&3&6&7&1&4&8}{7&3&6&8&1&4&9&2&5}{4&8&1&5&9&2&6&7&3}{3&6&8&7&2&5&4&9&1}{9&1&4&6&8&3&2&5&7}{5&7&2&1&4&9&8&3&6}
&\settablename{DS(9,35)}\namedsmallsudoku{1&4&7&2&5&8&3&6&9}{8&2&5&9&3&6&7&1&4}{6&9&3&4&7&1&5&8&2}{2&5&9&3&6&7&1&4&8}{7&3&6&8&1&4&9&2&5}{4&8&1&5&9&2&6&7&3}{3&7&4&1&8&5&2&9&6}{5&1&8&6&2&9&4&3&7}{9&6&2&7&4&3&8&5&1}
&\settablename{DS(9,36)}\namedsmallsudoku{1&4&7&2&5&8&3&6&9}{8&2&5&9&3&6&7&1&4}{6&9&3&4&7&1&5&8&2}{2&5&9&3&6&7&1&4&8}{7&3&6&8&1&4&9&2&5}{4&8&1&5&9&2&6&7&3}{3&7&4&6&2&9&8&5&1}{5&1&8&7&4&3&2&9&6}{9&6&2&1&8&5&4&3&7}
\\
\\
\settablename{DS(9,37)}\namedsmallsudoku{1&4&7&2&5&8&3&6&9}{8&2&5&9&3&6&7&1&4}{6&9&3&4&7&1&5&8&2}{2&5&9&3&6&7&8&4&1}{7&3&6&8&1&4&2&9&5}{4&8&1&5&9&2&6&3&7}{3&7&4&1&8&5&9&2&6}{5&1&8&6&2&9&4&7&3}{9&6&2&7&4&3&1&5&8}
&\settablename{DS(9,38)}\namedsmallsudoku{1&4&7&2&5&8&3&6&9}{8&2&5&9&3&6&7&1&4}{6&9&3&4&7&1&5&8&2}{2&5&9&3&8&4&6&7&1}{7&3&6&5&1&9&2&4&8}{4&8&1&7&6&2&9&3&5}{3&6&8&1&9&5&4&2&7}{9&1&4&6&2&7&8&5&3}{5&7&2&8&4&3&1&9&6}
&\settablename{DS(9,39)}\namedsmallsudoku{1&4&7&2&5&8&3&6&9}{8&2&5&9&3&6&7&1&4}{6&9&3&4&7&1&5&8&2}{2&5&9&3&8&4&6&7&1}{7&3&6&5&1&9&2&4&8}{4&8&1&7&6&2&9&3&5}{3&7&4&1&9&5&8&2&6}{5&1&8&6&2&7&4&9&3}{9&6&2&8&4&3&1&5&7}
&\settablename{DS(9,40)}\namedsmallsudoku{1&4&7&2&5&8&3&6&9}{8&2&5&9&3&6&7&1&4}{6&9&3&4&7&1&5&8&2}{2&5&9&8&1&4&6&7&3}{7&3&6&5&9&2&1&4&8}{4&8&1&3&6&7&9&2&5}{3&7&4&1&8&5&2&9&6}{5&1&8&6&2&9&4&3&7}{9&6&2&7&4&3&8&5&1}
\\
\\
\settablename{DS(9,41)}\namedsmallsudoku{1&4&7&2&5&8&3&6&9}{8&2&5&9&3&6&7&1&4}{6&9&3&4&7&1&5&8&2}{2&7&4&1&9&5&6&3&8}{5&3&8&6&2&7&9&4&1}{9&6&1&8&4&3&2&7&5}{3&8&6&5&1&9&4&2&7}{4&1&9&7&6&2&8&5&3}{7&5&2&3&8&4&1&9&6}
&\settablename{DS(9,42)}\namedsmallsudoku{1&4&7&2&5&8&3&9&6}{8&2&5&9&3&6&4&1&7}{6&9&3&4&7&1&8&5&2}{2&5&8&1&4&7&6&3&9}{9&3&6&8&2&5&7&4&1}{4&7&1&6&9&3&2&8&5}{3&6&9&5&8&2&1&7&4}{7&1&4&3&6&9&5&2&8}{5&8&2&7&1&4&9&6&3}
&\settablename{DS(9,43)}\namedsmallsudoku{1&4&7&2&5&8&3&9&6}{8&2&5&9&3&6&4&1&7}{6&9&3&4&7&1&8&5&2}{2&5&8&1&4&9&6&7&3}{9&3&6&7&2&5&1&4&8}{4&7&1&6&8&3&9&2&5}{3&6&9&5&1&7&2&8&4}{7&1&4&8&6&2&5&3&9}{5&8&2&3&9&4&7&6&1}
&\settablename{DS(9,44)}\namedsmallsudoku{1&4&7&2&5&8&3&9&6}{8&2&5&9&3&6&4&1&7}{6&9&3&4&7&1&8&5&2}{2&5&8&1&4&9&6&7&3}{9&3&6&7&2&5&1&4&8}{4&7&1&6&8&3&9&2&5}{3&8&4&5&1&7&2&6&9}{5&1&9&8&6&2&7&3&4}{7&6&2&3&9&4&5&8&1}
\\
\\
\settablename{DS(9,45)}\namedsmallsudoku{1&4&7&2&5&8&3&9&6}{8&2&5&9&3&6&4&1&7}{6&9&3&4&7&1&8&5&2}{2&5&8&1&4&9&6&7&3}{9&3&6&7&2&5&1&4&8}{4&7&1&6&8&3&9&2&5}{3&8&4&5&9&2&7&6&1}{5&1&9&3&6&7&2&8&4}{7&6&2&8&1&4&5&3&9}
&\settablename{DS(9,46)}\namedsmallsudoku{1&4&7&2&5&8&3&9&6}{8&2&5&9&3&6&4&1&7}{6&9&3&4&7&1&8&5&2}{2&5&8&1&4&9&7&6&3}{9&3&6&7&2&5&1&8&4}{4&7&1&6&8&3&5&2&9}{3&6&9&5&1&7&2&4&8}{7&1&4&8&6&2&9&3&5}{5&8&2&3&9&4&6&7&1}
&\settablename{DS(9,47)}\namedsmallsudoku{1&4&7&2&5&8&3&9&6}{8&2&5&9&3&6&4&1&7}{6&9&3&4&7&1&8&5&2}{2&5&8&1&4&9&7&6&3}{9&3&6&7&2&5&1&8&4}{4&7&1&6&8&3&5&2&9}{3&6&9&8&1&4&2&7&5}{7&1&4&5&9&2&6&3&8}{5&8&2&3&6&7&9&4&1}
&\settablename{DS(9,48)}\namedsmallsudoku{1&4&7&2&5&8&3&9&6}{8&2&5&9&3&6&4&1&7}{6&9&3&4&7&1&8&5&2}{2&5&8&1&4&9&7&6&3}{9&3&6&7&2&5&1&8&4}{4&7&1&6&8&3&5&2&9}{3&8&4&5&9&2&6&7&1}{5&1&9&3&6&7&2&4&8}{7&6&2&8&1&4&9&3&5}
\end{tabular}
\end{center}

\pagebreak

\begin{center}
\begin{tabular}{P{28mm}P{28mm}P{28mm}P{28mm}}
\settablename{DS(9,49)}\namedsmallsudoku{1&4&7&2&5&8&3&9&6}{8&2&5&9&3&6&4&1&7}{6&9&3&4&7&1&8&5&2}{2&5&8&1&6&9&7&3&4}{9&3&6&7&2&4&5&8&1}{4&7&1&5&8&3&2&6&9}{3&6&9&8&4&2&1&7&5}{7&1&4&3&9&5&6&2&8}{5&8&2&6&1&7&9&4&3}
&\settablename{DS(9,50)}\namedsmallsudoku{1&4&7&2&5&8&3&9&6}{8&2&5&9&3&6&4&1&7}{6&9&3&4&7&1&8&5&2}{2&5&8&1&6&9&7&3&4}{9&3&6&7&2&4&5&8&1}{4&7&1&5&8&3&2&6&9}{3&8&4&6&1&7&9&2&5}{5&1&9&8&4&2&6&7&3}{7&6&2&3&9&5&1&4&8}
&\settablename{DS(9,51)}\namedsmallsudoku{1&4&7&2&5&8&3&9&6}{8&2&5&9&3&6&4&1&7}{6&9&3&4&7&1&8&5&2}{2&5&8&1&6&9&7&3&4}{9&3&6&7&2&4&5&8&1}{4&7&1&5&8&3&2&6&9}{3&8&4&6&9&2&1&7&5}{5&1&9&3&4&7&6&2&8}{7&6&2&8&1&5&9&4&3}
&\settablename{DS(9,52)}\namedsmallsudoku{1&4&7&2&5&8&3&9&6}{8&2&5&9&3&6&4&1&7}{6&9&3&4&7&1&8&5&2}{2&5&8&1&6&9&7&4&3}{9&3&6&7&2&4&1&8&5}{4&7&1&5&8&3&6&2&9}{3&6&9&8&1&5&2&7&4}{7&1&4&6&9&2&5&3&8}{5&8&2&3&4&7&9&6&1}
\\
\\
\settablename{DS(9,53)}\namedsmallsudoku{1&4&7&2&5&8&3&9&6}{8&2&5&9&3&6&4&1&7}{6&9&3&4&7&1&8&5&2}{2&5&8&1&6&9&7&4&3}{9&3&6&7&2&4&1&8&5}{4&7&1&5&8&3&6&2&9}{3&6&9&8&4&2&5&7&1}{7&1&4&3&9&5&2&6&8}{5&8&2&6&1&7&9&3&4}
&\settablename{DS(9,54)}\namedsmallsudoku{1&4&7&2&5&8&3&9&6}{8&2&5&9&3&6&4&1&7}{6&9&3&4&7&1&8&5&2}{2&5&8&1&6&9&7&4&3}{9&3&6&7&2&4&1&8&5}{4&7&1&5&8&3&6&2&9}{3&8&4&6&9&2&5&7&1}{5&1&9&3&4&7&2&6&8}{7&6&2&8&1&5&9&3&4}
&\settablename{DS(9,55)}\namedsmallsudoku{1&4&7&2&5&8&3&9&6}{8&2&5&9&3&6&4&1&7}{6&9&3&4&7&1&8&5&2}{2&5&8&1&9&4&6&7&3}{9&3&6&5&2&7&1&4&8}{4&7&1&8&6&3&9&2&5}{3&6&9&7&1&5&2&8&4}{7&1&4&6&8&2&5&3&9}{5&8&2&3&4&9&7&6&1}
&\settablename{DS(9,56)}\namedsmallsudoku{1&4&7&2&5&8&3&9&6}{8&2&5&9&3&6&4&1&7}{6&9&3&4&7&1&8&5&2}{2&5&8&1&9&4&6&7&3}{9&3&6&5&2&7&1&4&8}{4&7&1&8&6&3&9&2&5}{3&6&9&7&4&2&5&8&1}{7&1&4&3&8&5&2&6&9}{5&8&2&6&1&9&7&3&4}
\\
\\
\settablename{DS(9,57)}\namedsmallsudoku{1&4&7&2&5&8&3&9&6}{8&2&5&9&3&6&4&1&7}{6&9&3&4&7&1&8&5&2}{2&5&8&1&9&4&6&7&3}{9&3&6&5&2&7&1&4&8}{4&7&1&8&6&3&9&2&5}{3&8&4&7&1&5&2&6&9}{5&1&9&6&8&2&7&3&4}{7&6&2&3&4&9&5&8&1}
&\settablename{DS(9,58)}\namedsmallsudoku{1&4&7&2&5&8&3&9&6}{8&2&5&9&3&6&4&1&7}{6&9&3&4&7&1&8&5&2}{2&5&8&1&9&4&7&6&3}{9&3&6&5&2&7&1&8&4}{4&7&1&8&6&3&5&2&9}{3&6&9&7&1&5&2&4&8}{7&1&4&6&8&2&9&3&5}{5&8&2&3&4&9&6&7&1}
&\settablename{DS(9,59)}\namedsmallsudoku{1&4&7&2&5&8&3&9&6}{8&2&5&9&3&6&4&1&7}{6&9&3&4&7&1&8&5&2}{2&5&8&1&9&4&7&6&3}{9&3&6&5&2&7&1&8&4}{4&7&1&8&6&3&5&2&9}{3&8&4&6&1&9&2&7&5}{5&1&9&7&4&2&6&3&8}{7&6&2&3&8&5&9&4&1}
&\settablename{DS(9,60)}\namedsmallsudoku{1&4&7&2&5&8&3&9&6}{8&2&5&9&3&6&4&1&7}{6&9&3&4&7&1&8&5&2}{2&5&8&3&6&9&1&7&4}{9&3&6&7&1&4&5&2&8}{4&7&1&5&8&2&9&6&3}{3&6&9&1&4&7&2&8&5}{7&1&4&8&2&5&6&3&9}{5&8&2&6&9&3&7&4&1}
\\
\\
\settablename{DS(9,61)}\namedsmallsudoku{1&4&7&2&5&8&3&9&6}{8&2&5&9&3&6&4&1&7}{6&9&3&4&7&1&8&5&2}{2&5&8&3&6&9&1&7&4}{9&3&6&7&1&4&5&2&8}{4&7&1&5&8&2&9&6&3}{3&6&9&8&2&5&7&4&1}{7&1&4&6&9&3&2&8&5}{5&8&2&1&4&7&6&3&9}
&\settablename{DS(9,62)}\namedsmallsudoku{1&4&7&2&5&8&3&9&6}{8&2&5&9&3&6&4&1&7}{6&9&3&4&7&1&8&5&2}{2&5&8&3&6&9&7&4&1}{9&3&6&7&1&4&2&8&5}{4&7&1&5&8&2&6&3&9}{3&6&9&1&4&7&5&2&8}{7&1&4&8&2&5&9&6&3}{5&8&2&6&9&3&1&7&4}
&\settablename{DS(9,63)}\namedsmallsudoku{1&4&7&2&5&8&3&9&6}{8&2&5&9&3&6&4&1&7}{6&9&3&4&7&1&8&5&2}{2&5&8&3&6&9&7&4&1}{9&3&6&7&1&4&2&8&5}{4&7&1&5&8&2&6&3&9}{3&6&9&8&2&5&1&7&4}{7&1&4&6&9&3&5&2&8}{5&8&2&1&4&7&9&6&3}
&\settablename{DS(9,64)}\namedsmallsudoku{1&4&7&2&5&8&3&9&6}{8&2&5&9&3&6&4&1&7}{6&9&3&4&7&1&8&5&2}{2&5&8&3&9&4&6&7&1}{9&3&6&5&1&7&2&4&8}{4&7&1&8&6&2&9&3&5}{3&6&9&1&8&5&7&2&4}{7&1&4&6&2&9&5&8&3}{5&8&2&7&4&3&1&6&9}
\\
\\
\settablename{DS(9,65)}\namedsmallsudoku{1&4&7&2&5&8&3&9&6}{8&2&5&9&3&6&4&1&7}{6&9&3&4&7&1&8&5&2}{2&5&8&3&9&4&6&7&1}{9&3&6&5&1&7&2&4&8}{4&7&1&8&6&2&9&3&5}{3&6&9&7&2&5&1&8&4}{7&1&4&6&8&3&5&2&9}{5&8&2&1&4&9&7&6&3}
&\settablename{DS(9,66)}\namedsmallsudoku{1&4&7&2&5&8&3&9&6}{8&2&5&9&3&6&4&1&7}{6&9&3&4&7&1&8&5&2}{2&5&8&3&9&4&6&7&1}{9&3&6&5&1&7&2&4&8}{4&7&1&8&6&2&9&3&5}{3&8&4&7&2&5&1&6&9}{5&1&9&6&8&3&7&2&4}{7&6&2&1&4&9&5&8&3}
&\settablename{DS(9,67)}\namedsmallsudoku{1&4&7&2&5&8&3&9&6}{8&2&5&9&3&6&4&1&7}{6&9&3&4&7&1&8&5&2}{2&5&8&3&9&4&7&6&1}{9&3&6&5&1&7&2&8&4}{4&7&1&8&6&2&5&3&9}{3&6&9&7&2&5&1&4&8}{7&1&4&6&8&3&9&2&5}{5&8&2&1&4&9&6&7&3}
&\settablename{DS(9,68)}\namedsmallsudoku{1&4&7&2&5&8&3&9&6}{8&2&5&9&3&6&4&1&7}{6&9&3&4&7&1&8&5&2}{2&5&8&3&9&4&7&6&1}{9&3&6&5&1&7&2&8&4}{4&7&1&8&6&2&5&3&9}{3&8&4&6&2&9&1&7&5}{5&1&9&7&4&3&6&2&8}{7&6&2&1&8&5&9&4&3}
\\
\\
\settablename{DS(9,69)}\namedsmallsudoku{1&4&7&2&5&8&3&9&6}{8&2&5&9&3&6&4&1&7}{6&9&3&4&7&1&8&5&2}{2&5&8&6&9&3&1&7&4}{9&3&6&1&4&7&5&2&8}{4&7&1&8&2&5&9&6&3}{3&6&9&7&1&4&2&8&5}{7&1&4&5&8&2&6&3&9}{5&8&2&3&6&9&7&4&1}
&\settablename{DS(9,70)}\namedsmallsudoku{1&4&7&2&5&8&3&9&6}{8&2&5&9&3&6&4&1&7}{6&9&3&4&7&1&8&5&2}{2&5&9&1&8&4&6&7&3}{7&3&6&5&2&9&1&4&8}{4&8&1&7&6&3&9&2&5}{3&7&4&6&9&2&5&8&1}{5&1&8&3&4&7&2&6&9}{9&6&2&8&1&5&7&3&4}
&\settablename{DS(9,71)}\namedsmallsudoku{1&4&7&2&5&8&3&9&6}{8&2&5&9&3&6&4&1&7}{6&9&3&4&7&1&8&5&2}{2&5&9&1&8&4&6&7&3}{7&3&6&5&2&9&1&4&8}{4&8&1&7&6&3&9&2&5}{3&7&4&8&1&5&2&6&9}{5&1&8&6&9&2&7&3&4}{9&6&2&3&4&7&5&8&1}
&\settablename{DS(9,72)}\namedsmallsudoku{1&4&7&2&5&8&3&9&6}{8&2&5&9&3&6&4&1&7}{6&9&3&4&7&1&8&5&2}{2&5&9&3&4&7&1&6&8}{7&3&6&8&1&5&9&2&4}{4&8&1&6&9&2&5&7&3}{3&6&8&5&2&9&7&4&1}{9&1&4&7&6&3&2&8&5}{5&7&2&1&8&4&6&3&9}
\end{tabular}
\end{center}

\pagebreak

\begin{center}
\begin{tabular}{P{28mm}P{28mm}P{28mm}P{28mm}}
\settablename{DS(9,73)}\namedsmallsudoku{1&4&7&2&5&8&3&9&6}{8&2&5&9&3&6&4&1&7}{6&9&3&4&7&1&8&5&2}{2&5&9&3&4&7&1&6&8}{7&3&6&8&1&5&9&2&4}{4&8&1&6&9&2&5&7&3}{3&7&4&1&6&9&2&8&5}{5&1&8&7&2&4&6&3&9}{9&6&2&5&8&3&7&4&1}
&\settablename{DS(9,74)}\namedsmallsudoku{1&4&7&2&5&8&3&9&6}{8&2&5&9&3&6&4&1&7}{6&9&3&4&7&1&8&5&2}{2&5&9&3&4&7&1&6&8}{7&3&6&8&1&5&9&2&4}{4&8&1&6&9&2&5&7&3}{3&7&4&5&2&9&6&8&1}{5&1&8&7&6&3&2&4&9}{9&6&2&1&8&4&7&3&5}
&\settablename{DS(9,75)}\namedsmallsudoku{1&4&7&2&5&8&3&9&6}{8&2&5&9&3&6&4&1&7}{6&9&3&4&7&1&8&5&2}{2&5&9&3&6&7&1&8&4}{7&3&6&8&1&4&5&2&9}{4&8&1&5&9&2&7&6&3}{3&6&8&1&4&9&2&7&5}{9&1&4&7&2&5&6&3&8}{5&7&2&6&8&3&9&4&1}
&\settablename{DS(9,76)}\namedsmallsudoku{1&4&7&2&5&8&3&9&6}{8&2&5&9&3&6&4&1&7}{6&9&3&4&7&1&8&5&2}{2&5&9&3&6&7&1&8&4}{7&3&6&8&1&4&5&2&9}{4&8&1&5&9&2&7&6&3}{3&6&8&7&2&5&9&4&1}{9&1&4&6&8&3&2&7&5}{5&7&2&1&4&9&6&3&8}
\\
\\
\settablename{DS(9,77)}\namedsmallsudoku{1&4&7&2&5&8&3&9&6}{8&2&5&9&3&6&4&1&7}{6&9&3&4&7&1&8&5&2}{2&5&9&3&8&4&6&7&1}{7&3&6&5&1&9&2&4&8}{4&8&1&7&6&2&9&3&5}{3&6&8&1&9&5&7&2&4}{9&1&4&6&2&7&5&8&3}{5&7&2&8&4&3&1&6&9}
&\settablename{DS(9,78)}\namedsmallsudoku{1&4&7&2&5&8&3&9&6}{8&2&5&9&3&6&4&1&7}{6&9&3&4&7&1&8&5&2}{2&5&9&3&8&4&7&6&1}{7&3&6&5&1&9&2&8&4}{4&8&1&7&6&2&5&3&9}{3&7&4&1&9&5&6&2&8}{5&1&8&6&2&7&9&4&3}{9&6&2&8&4&3&1&7&5}
&\settablename{DS(9,79)}\namedsmallsudoku{1&4&7&2&5&8&3&9&6}{8&2&5&9&3&6&4&1&7}{6&9&3&4&7&1&8&5&2}{2&5&9&6&8&3&1&7&4}{7&3&6&1&4&9&5&2&8}{4&8&1&7&2&5&9&6&3}{3&6&8&5&1&7&2&4&9}{9&1&4&8&6&2&7&3&5}{5&7&2&3&9&4&6&8&1}
&\settablename{DS(9,80)}\namedsmallsudoku{1&4&7&2&5&8&3&9&6}{8&2&5&9&3&6&4&1&7}{6&9&3&4&7&1&8&5&2}{2&5&9&6&8&3&1&7&4}{7&3&6&1&4&9&5&2&8}{4&8&1&7&2&5&9&6&3}{3&6&8&5&9&2&7&4&1}{9&1&4&3&6&7&2&8&5}{5&7&2&8&1&4&6&3&9}
\\
\\
\settablename{DS(9,81)}\namedsmallsudoku{1&4&7&2&5&8&3&9&6}{8&2&5&9&3&6&4&1&7}{6&9&3&4&7&1&8&5&2}{2&5&9&6&8&3&1&7&4}{7&3&6&1&4&9&5&2&8}{4&8&1&7&2&5&9&6&3}{3&7&4&5&9&2&6&8&1}{5&1&8&3&6&7&2&4&9}{9&6&2&8&1&4&7&3&5}
&\settablename{DS(9,82)}\namedsmallsudoku{1&4&7&2&5&8&3&9&6}{8&2&5&9&3&6&4&1&7}{6&9&3&4&7&1&8&5&2}{2&5&9&7&1&4&6&3&8}{7&3&6&5&8&2&9&4&1}{4&8&1&3&6&9&2&7&5}{3&6&8&1&9&5&7&2&4}{9&1&4&6&2&7&5&8&3}{5&7&2&8&4&3&1&6&9}
&\settablename{DS(9,83)}\namedsmallsudoku{1&4&7&2&5&8&3&9&6}{8&2&5&9&3&6&4&1&7}{6&9&3&4&7&1&8&5&2}{2&5&9&7&1&4&6&8&3}{7&3&6&5&8&2&1&4&9}{4&8&1&3&6&9&7&2&5}{3&6&8&1&9&5&2&7&4}{9&1&4&6&2&7&5&3&8}{5&7&2&8&4&3&9&6&1}
&\settablename{DS(9,84)}\namedsmallsudoku{1&4&7&2&5&8&3&9&6}{8&2&5&9&3&6&4&1&7}{6&9&3&4&7&1&8&5&2}{2&5&9&7&1&4&6&8&3}{7&3&6&5&8&2&1&4&9}{4&8&1&3&6&9&7&2&5}{3&7&4&1&9&5&2&6&8}{5&1&8&6&2&7&9&3&4}{9&6&2&8&4&3&5&7&1}
\\
\\
\settablename{DS(9,85)}\namedsmallsudoku{1&4&7&2&5&8&3&9&6}{8&2&5&9&3&6&4&1&7}{6&9&3&4&7&1&8&5&2}{2&5&9&7&4&3&1&6&8}{7&3&6&1&8&5&9&2&4}{4&8&1&6&2&9&5&7&3}{3&6&8&5&1&7&2&4&9}{9&1&4&8&6&2&7&3&5}{5&7&2&3&9&4&6&8&1}
&\settablename{DS(9,86)}\namedsmallsudoku{1&4&7&2&5&8&3&9&6}{8&2&5&9&3&6&4&1&7}{6&9&3&4&7&1&8&5&2}{2&5&9&7&4&3&1&6&8}{7&3&6&1&8&5&9&2&4}{4&8&1&6&2&9&5&7&3}{3&6&8&5&9&2&7&4&1}{9&1&4&3&6&7&2&8&5}{5&7&2&8&1&4&6&3&9}
&\settablename{DS(9,87)}\namedsmallsudoku{1&4&7&2&5&8&3&9&6}{8&2&5&9&3&6&4&1&7}{6&9&3&4&7&1&8&5&2}{2&5&9&7&4&3&1&6&8}{7&3&6&1&8&5&9&2&4}{4&8&1&6&2&9&5&7&3}{3&7&4&5&9&2&6&8&1}{5&1&8&3&6&7&2&4&9}{9&6&2&8&1&4&7&3&5}
&\settablename{DS(9,88)}\namedsmallsudoku{1&4&7&2&5&8&3&9&6}{8&2&5&9&3&6&4&1&7}{6&9&3&4&7&1&8&5&2}{2&5&9&7&4&3&6&8&1}{7&3&6&1&8&5&2&4&9}{4&8&1&6&2&9&7&3&5}{3&6&8&5&1&7&9&2&4}{9&1&4&8&6&2&5&7&3}{5&7&2&3&9&4&1&6&8}
\\
\\
\settablename{DS(9,89)}\namedsmallsudoku{1&4&7&2&5&8&3&9&6}{8&2&5&9&3&6&4&1&7}{6&9&3&4&7&1&8&5&2}{2&5&9&7&4&3&6&8&1}{7&3&6&1&8&5&2&4&9}{4&8&1&6&2&9&7&3&5}{3&6&8&5&9&2&1&7&4}{9&1&4&3&6&7&5&2&8}{5&7&2&8&1&4&9&6&3}
&\settablename{DS(9,90)}\namedsmallsudoku{1&4&7&2&5&8&3&9&6}{8&2&5&9&3&6&4&1&7}{6&9&3&4&7&1&8&5&2}{2&5&9&7&4&3&6&8&1}{7&3&6&1&8&5&2&4&9}{4&8&1&6&2&9&7&3&5}{3&7&4&5&9&2&1&6&8}{5&1&8&3&6&7&9&2&4}{9&6&2&8&1&4&5&7&3}
&\settablename{DS(9,91)}\namedsmallsudoku{1&4&7&2&5&8&3&9&6}{8&2&5&9&3&6&4&1&7}{6&9&3&4&7&1&8&5&2}{2&5&9&7&6&3&1&4&8}{7&3&6&1&8&4&9&2&5}{4&8&1&5&2&9&6&7&3}{3&7&4&8&1&5&2&6&9}{5&1&8&6&9&2&7&3&4}{9&6&2&3&4&7&5&8&1}
&\settablename{DS(9,92)}\namedsmallsudoku{1&4&7&2&5&8&3&9&6}{8&2&5&9&3&6&4&1&7}{6&9&3&4&7&1&8&5&2}{2&5&9&8&4&3&6&7&1}{7&3&6&1&9&5&2&4&8}{4&8&1&6&2&7&9&3&5}{3&6&8&5&1&9&7&2&4}{9&1&4&7&6&2&5&8&3}{5&7&2&3&8&4&1&6&9}
\\
\\
\settablename{DS(9,93)}\namedsmallsudoku{1&4&7&2&5&8&3&9&6}{8&2&5&9&3&6&4&1&7}{6&9&3&4&7&1&8&5&2}{2&5&9&8&4&3&6&7&1}{7&3&6&1&9&5&2&4&8}{4&8&1&6&2&7&9&3&5}{3&6&8&7&1&4&5&2&9}{9&1&4&5&8&2&7&6&3}{5&7&2&3&6&9&1&8&4}
&\settablename{DS(9,94)}\namedsmallsudoku{1&4&7&2&5&8&3&9&6}{8&2&5&9&3&6&4&1&7}{6&9&3&4&7&1&8&5&2}{2&5&9&8&4&3&6&7&1}{7&3&6&1&9&5&2&4&8}{4&8&1&6&2&7&9&3&5}{3&7&4&5&8&2&1&6&9}{5&1&8&3&6&9&7&2&4}{9&6&2&7&1&4&5&8&3}
&\settablename{DS(9,95)}\namedsmallsudoku{1&4&7&2&5&8&3&9&6}{8&2&5&9&3&6&4&1&7}{6&9&3&4&7&1&8&5&2}{2&5&9&8&4&3&7&6&1}{7&3&6&1&9&5&2&8&4}{4&8&1&6&2&7&5&3&9}{3&6&8&7&1&4&9&2&5}{9&1&4&5&8&2&6&7&3}{5&7&2&3&6&9&1&4&8}
&\settablename{DS(9,96)}\namedsmallsudoku{1&4&7&2&5&8&3&9&6}{8&2&5&9&3&6&4&1&7}{6&9&3&4&7&1&8&5&2}{2&5&9&8&4&3&7&6&1}{7&3&6&1&9&5&2&8&4}{4&8&1&6&2&7&5&3&9}{3&7&4&5&1&9&6&2&8}{5&1&8&7&6&2&9&4&3}{9&6&2&3&8&4&1&7&5}
\end{tabular}
\end{center}

\pagebreak

\begin{center}
\begin{tabular}{P{28mm}P{28mm}P{28mm}P{28mm}}
\settablename{DS(9,97)}\namedsmallsudoku{1&4&7&2&5&8&3&9&6}{8&2&5&9&3&6&4&1&7}{6&9&3&4&7&1&8&5&2}{2&5&9&8&6&3&1&7&4}{7&3&6&1&9&4&5&2&8}{4&8&1&5&2&7&9&6&3}{3&6&8&7&1&5&2&4&9}{9&1&4&6&8&2&7&3&5}{5&7&2&3&4&9&6&8&1}
&\settablename{DS(9,98)}\namedsmallsudoku{1&4&7&2&5&8&3&9&6}{8&2&5&9&3&6&4&1&7}{6&9&3&4&7&1&8&5&2}{2&5&9&8&6&3&1&7&4}{7&3&6&1&9&4&5&2&8}{4&8&1&5&2&7&9&6&3}{3&7&4&6&1&9&2&8&5}{5&1&8&7&4&2&6&3&9}{9&6&2&3&8&5&7&4&1}
&\settablename{DS(9,99)}\namedsmallsudoku{1&4&7&2&5&8&3&9&6}{8&2&5&9&3&6&4&1&7}{6&9&3&4&7&1&8&5&2}{2&5&9&8&6&3&7&4&1}{7&3&6&1&9&4&2&8&5}{4&8&1&5&2&7&6&3&9}{3&6&8&7&1&5&9&2&4}{9&1&4&6&8&2&5&7&3}{5&7&2&3&4&9&1&6&8}
&\settablename{DS(9,100)}\namedsmallsudoku{1&4&7&2&5&8&3&9&6}{8&2&5&9&3&6&4&1&7}{6&9&3&4&7&1&8&5&2}{2&5&9&8&6&3&7&4&1}{7&3&6&1&9&4&2&8&5}{4&8&1&5&2&7&6&3&9}{3&7&4&6&1&9&5&2&8}{5&1&8&7&4&2&9&6&3}{9&6&2&3&8&5&1&7&4}
\\
\\
\settablename{DS(9,101)}\namedsmallsudoku{1&4&7&2&5&8&3&9&6}{8&2&5&9&3&6&4&1&7}{6&9&3&4&7&1&8&5&2}{2&6&9&3&8&4&1&7&5}{7&3&4&5&1&9&6&2&8}{5&8&1&7&6&2&9&4&3}{3&7&6&1&9&5&2&8&4}{4&1&8&6&2&7&5&3&9}{9&5&2&8&4&3&7&6&1}
&\settablename{DS(9,102)}\namedsmallsudoku{1&4&7&2&5&8&3&9&6}{8&2&5&9&3&6&4&1&7}{6&9&3&4&7&1&8&5&2}{2&6&9&3&8&5&1&7&4}{7&3&4&6&1&9&5&2&8}{5&8&1&7&4&2&9&6&3}{3&7&6&1&9&4&2&8&5}{4&1&8&5&2&7&6&3&9}{9&5&2&8&6&3&7&4&1}
&\settablename{DS(9,103)}\namedsmallsudoku{1&4&7&2&5&8&3&9&6}{8&2&5&9&3&6&4&1&7}{6&9&3&4&7&1&8&5&2}{2&6&9&3&8&5&7&4&1}{7&3&4&6&1&9&2&8&5}{5&8&1&7&4&2&6&3&9}{3&7&6&1&9&4&5&2&8}{4&1&8&5&2&7&9&6&3}{9&5&2&8&6&3&1&7&4}
&\settablename{DS(9,104)}\namedsmallsudoku{1&4&7&2&5&8&3&9&6}{8&2&5&9&3&6&4&1&7}{6&9&3&4&7&1&8&5&2}{2&6&9&7&1&4&5&8&3}{7&3&4&5&8&2&1&6&9}{5&8&1&3&6&9&7&2&4}{3&7&6&1&9&5&2&4&8}{4&1&8&6&2&7&9&3&5}{9&5&2&8&4&3&6&7&1}
\\
\\
\settablename{DS(9,105)}\namedsmallsudoku{1&4&7&2&5&8&3&9&6}{8&2&5&9&3&6&4&1&7}{6&9&3&4&7&1&8&5&2}{2&6&9&7&4&3&5&8&1}{7&3&4&1&8&5&2&6&9}{5&8&1&6&2&9&7&3&4}{3&7&6&8&1&4&9&2&5}{4&1&8&5&9&2&6&7&3}{9&5&2&3&6&7&1&4&8}
&\settablename{DS(9,106)}\namedsmallsudoku{1&4&7&2&5&8&3&9&6}{8&2&5&9&3&6&4&1&7}{6&9&3&4&7&1&8&5&2}{2&6&9&8&1&4&5&7&3}{7&3&4&5&9&2&1&6&8}{5&8&1&3&6&7&9&2&4}{3&7&6&1&8&5&2&4&9}{4&1&8&6&2&9&7&3&5}{9&5&2&7&4&3&6&8&1}
&\settablename{DS(9,107)}\namedsmallsudoku{1&4&7&2&5&8&3&9&6}{8&2&5&9&3&6&4&1&7}{6&9&3&4&7&1&8&5&2}{2&6&9&8&4&3&1&7&5}{7&3&4&1&9&5&6&2&8}{5&8&1&6&2&7&9&4&3}{3&7&6&5&1&9&2&8&4}{4&1&8&7&6&2&5&3&9}{9&5&2&3&8&4&7&6&1}
&\settablename{DS(9,108)}\namedsmallsudoku{1&4&7&2&5&8&3&9&6}{8&2&5&9&3&6&4&1&7}{6&9&3&4&7&1&8&5&2}{2&6&9&8&4&3&5&7&1}{7&3&4&1&9&5&2&6&8}{5&8&1&6&2&7&9&3&4}{3&7&6&5&8&2&1&4&9}{4&1&8&3&6&9&7&2&5}{9&5&2&7&1&4&6&8&3}
\\
\\
\settablename{DS(9,109)}\namedsmallsudoku{1&4&7&2&5&8&3&9&6}{8&2&5&9&3&6&4&1&7}{6&9&3&4&7&1&8&5&2}{2&7&4&1&8&5&6&3&9}{5&3&8&6&2&9&7&4&1}{9&6&1&7&4&3&2&8&5}{3&8&6&5&1&7&9&2&4}{4&1&9&8&6&2&5&7&3}{7&5&2&3&9&4&1&6&8}
&\settablename{DS(9,110)}\namedsmallsudoku{1&4&7&2&5&8&3&9&6}{8&2&5&9&3&6&4&1&7}{6&9&3&4&7&1&8&5&2}{2&7&4&1&9&5&6&3&8}{5&3&8&6&2&7&9&4&1}{9&6&1&8&4&3&2&7&5}{3&8&6&5&1&9&7&2&4}{4&1&9&7&6&2&5&8&3}{7&5&2&3&8&4&1&6&9}
&\settablename{DS(9,111)}\namedsmallsudoku{1&4&7&2&5&8&3&9&6}{8&2&5&9&3&6&4&1&7}{6&9&3&4&7&1&8&5&2}{2&7&4&1&9&5&6&3&8}{5&3&8&6&2&7&9&4&1}{9&6&1&8&4&3&2&7&5}{3&8&6&7&1&4&5&2&9}{4&1&9&5&8&2&7&6&3}{7&5&2&3&6&9&1&8&4}
&\settablename{DS(9,112)}\namedsmallsudoku{1&4&7&2&5&8&3&9&6}{8&2&5&9&3&6&4&1&7}{6&9&3&4&7&1&8&5&2}{2&7&4&1&9&5&6&8&3}{5&3&8&6&2&7&1&4&9}{9&6&1&8&4&3&7&2&5}{3&8&6&5&1&9&2&7&4}{4&1&9&7&6&2&5&3&8}{7&5&2&3&8&4&9&6&1}
\\
\\
\settablename{DS(9,113)}\namedsmallsudoku{1&4&7&2&5&8&3&9&6}{8&2&5&9&3&6&4&1&7}{6&9&3&4&7&1&8&5&2}{2&7&4&3&6&9&1&8&5}{5&3&8&7&1&4&6&2&9}{9&6&1&5&8&2&7&4&3}{3&8&6&1&9&5&2&7&4}{4&1&9&6&2&7&5&3&8}{7&5&2&8&4&3&9&6&1}
&\settablename{DS(9,114)}\namedsmallsudoku{1&4&7&2&5&8&3&9&6}{8&2&5&9&3&6&4&1&7}{6&9&3&4&7&1&8&5&2}{2&7&4&3&9&5&1&6&8}{5&3&8&6&1&7&9&2&4}{9&6&1&8&4&2&5&7&3}{3&8&6&5&2&9&7&4&1}{4&1&9&7&6&3&2&8&5}{7&5&2&1&8&4&6&3&9}
&\settablename{DS(9,115)}\namedsmallsudoku{1&4&7&2&5&8&3&9&6}{8&2&5&9&3&6&4&1&7}{6&9&3&4&7&1&8&5&2}{2&7&4&5&1&9&6&3&8}{5&3&8&7&6&2&9&4&1}{9&6&1&3&8&4&2&7&5}{3&8&6&1&9&5&7&2&4}{4&1&9&6&2&7&5&8&3}{7&5&2&8&4&3&1&6&9}
&\settablename{DS(9,116)}\namedsmallsudoku{1&4&7&2&5&8&3&9&6}{8&2&5&9&3&6&4&1&7}{6&9&3&4&7&1&8&5&2}{2&7&4&5&1&9&6&8&3}{5&3&8&7&6&2&1&4&9}{9&6&1&3&8&4&7&2&5}{3&8&6&1&9&5&2&7&4}{4&1&9&6&2&7&5&3&8}{7&5&2&8&4&3&9&6&1}
\\
\\
\settablename{DS(9,117)}\namedsmallsudoku{1&4&7&2&5&9&3&6&8}{8&2&5&7&3&6&9&1&4}{6&9&3&4&8&1&5&7&2}{2&5&8&3&6&7&1&4&9}{9&3&6&8&1&4&7&2&5}{4&7&1&5&9&2&6&8&3}{3&6&9&1&4&8&2&5&7}{7&1&4&9&2&5&8&3&6}{5&8&2&6&7&3&4&9&1}
&\settablename{DS(9,118)}\namedsmallsudoku{1&4&7&2&5&9&3&6&8}{8&2&5&7&3&6&9&1&4}{6&9&3&4&8&1&5&7&2}{2&5&8&3&6&7&1&4&9}{9&3&6&8&1&4&7&2&5}{4&7&1&5&9&2&6&8&3}{3&8&4&1&7&5&2&9&6}{5&1&9&6&2&8&4&3&7}{7&6&2&9&4&3&8&5&1}
&\settablename{DS(9,119)}\namedsmallsudoku{1&4&7&2&5&9&3&6&8}{8&2&5&7&3&6&9&1&4}{6&9&3&4&8&1&5&7&2}{2&5&8&3&6&7&4&9&1}{9&3&6&8&1&4&2&5&7}{4&7&1&5&9&2&8&3&6}{3&6&9&1&4&8&7&2&5}{7&1&4&9&2&5&6&8&3}{5&8&2&6&7&3&1&4&9}
&\settablename{DS(9,120)}\namedsmallsudoku{1&4&7&2&5&9&3&6&8}{8&2&5&7&3&6&9&1&4}{6&9&3&4&8&1&5&7&2}{2&5&8&3&6&7&4&9&1}{9&3&6&8&1&4&2&5&7}{4&7&1&5&9&2&8&3&6}{3&8&4&1&7&5&6&2&9}{5&1&9&6&2&8&7&4&3}{7&6&2&9&4&3&1&8&5}
\end{tabular}
\end{center}

\pagebreak

\begin{center}
\begin{tabular}{P{28mm}P{28mm}P{28mm}P{28mm}}
\settablename{DS(9,121)}\namedsmallsudoku{1&4&7&2&5&9&3&6&8}{8&2&5&7&3&6&9&1&4}{6&9&3&4&8&1&5&7&2}{2&5&9&1&6&7&4&8&3}{7&3&6&8&2&4&1&5&9}{4&8&1&5&9&3&7&2&6}{3&6&8&9&1&5&2&4&7}{9&1&4&6&7&2&8&3&5}{5&7&2&3&4&8&6&9&1}
&\settablename{DS(9,122)}\namedsmallsudoku{1&4&7&2&5&9&3&6&8}{8&2&5&7&3&6&9&1&4}{6&9&3&4&8&1&5&7&2}{2&5&9&1&6&7&4&8&3}{7&3&6&8&2&4&1&5&9}{4&8&1&5&9&3&7&2&6}{3&7&4&6&1&8&2&9&5}{5&1&8&9&4&2&6&3&7}{9&6&2&3&7&5&8&4&1}
&\settablename{DS(9,123)}\namedsmallsudoku{1&4&7&2&5&9&3&6&8}{8&2&5&7&3&6&9&1&4}{6&9&3&4&8&1&5&7&2}{2&5&9&1&6&7&8&4&3}{7&3&6&8&2&4&1&9&5}{4&8&1&5&9&3&6&2&7}{3&6&8&9&4&2&7&5&1}{9&1&4&3&7&5&2&8&6}{5&7&2&6&1&8&4&3&9}
&\settablename{DS(9,124)}\namedsmallsudoku{1&4&7&2&5&9&3&6&8}{8&2&5&7&3&6&9&1&4}{6&9&3&4&8&1&5&7&2}{2&5&9&1&6&7&8&4&3}{7&3&6&8&2&4&1&9&5}{4&8&1&5&9&3&6&2&7}{3&7&4&6&1&8&2&5&9}{5&1&8&9&4&2&7&3&6}{9&6&2&3&7&5&4&8&1}
\\
\\
\settablename{DS(9,125)}\namedsmallsudoku{1&4&7&2&5&9&3&6&8}{8&2&5&7&3&6&9&1&4}{6&9&3&4&8&1&5&7&2}{2&5&9&1&6&7&8&4&3}{7&3&6&8&2&4&1&9&5}{4&8&1&5&9&3&6&2&7}{3&7&4&9&1&5&2&8&6}{5&1&8&6&7&2&4&3&9}{9&6&2&3&4&8&7&5&1}
&\settablename{DS(9,126)}\namedsmallsudoku{1&4&7&2&5&9&3&6&8}{8&2&5&7&3&6&9&1&4}{6&9&3&4&8&1&5&7&2}{2&5&9&3&4&7&1&8&6}{7&3&6&8&1&5&4&2&9}{4&8&1&6&9&2&7&5&3}{3&6&8&1&7&4&2&9&5}{9&1&4&5&2&8&6&3&7}{5&7&2&9&6&3&8&4&1}
&\settablename{DS(9,127)}\namedsmallsudoku{1&4&7&2&5&9&3&6&8}{8&2&5&7&3&6&9&1&4}{6&9&3&4&8&1&5&7&2}{2&5&9&3&4&7&1&8&6}{7&3&6&8&1&5&4&2&9}{4&8&1&6&9&2&7&5&3}{3&7&4&1&6&8&2&9&5}{5&1&8&9&2&4&6&3&7}{9&6&2&5&7&3&8&4&1}
&\settablename{DS(9,128)}\namedsmallsudoku{1&4&7&2&5&9&3&6&8}{8&2&5&7&3&6&9&1&4}{6&9&3&4&8&1&5&7&2}{2&5&9&3&4&7&1&8&6}{7&3&6&8&1&5&4&2&9}{4&8&1&6&9&2&7&5&3}{3&7&4&5&2&8&6&9&1}{5&1&8&9&6&3&2&4&7}{9&6&2&1&7&4&8&3&5}
\\
\\
\settablename{DS(9,129)}\namedsmallsudoku{1&4&7&2&5&9&3&6&8}{8&2&5&7&3&6&9&1&4}{6&9&3&4&8&1&5&7&2}{2&5&9&3&4&7&6&8&1}{7&3&6&8&1&5&2&4&9}{4&8&1&6&9&2&7&3&5}{3&6&8&9&2&4&1&5&7}{9&1&4&5&7&3&8&2&6}{5&7&2&1&6&8&4&9&3}
&\settablename{DS(9,130)}\namedsmallsudoku{1&4&7&2&5&9&3&6&8}{8&2&5&7&3&6&9&1&4}{6&9&3&4&8&1&5&7&2}{2&5&9&3&4&7&6&8&1}{7&3&6&8&1&5&2&4&9}{4&8&1&6&9&2&7&3&5}{3&7&4&5&2&8&1&9&6}{5&1&8&9&6&3&4&2&7}{9&6&2&1&7&4&8&5&3}
&\settablename{DS(9,131)}\namedsmallsudoku{1&4&7&2&5&9&3&6&8}{8&2&5&7&3&6&9&1&4}{6&9&3&4&8&1&5&7&2}{2&5&9&3&6&8&1&4&7}{7&3&6&9&1&4&8&2&5}{4&8&1&5&7&2&6&9&3}{3&6&8&1&4&7&2&5&9}{9&1&4&8&2&5&7&3&6}{5&7&2&6&9&3&4&8&1}
&\settablename{DS(9,132)}\namedsmallsudoku{1&4&7&2&5&9&3&6&8}{8&2&5&7&3&6&9&1&4}{6&9&3&4&8&1&5&7&2}{2&5&9&3&6&8&1&4&7}{7&3&6&9&1&4&8&2&5}{4&8&1&5&7&2&6&9&3}{3&7&4&1&9&5&2&8&6}{5&1&8&6&2&7&4&3&9}{9&6&2&8&4&3&7&5&1}
\\
\\
\settablename{DS(9,133)}\namedsmallsudoku{1&4&7&2&5&9&3&6&8}{8&2&5&7&3&6&9&1&4}{6&9&3&4&8&1&5&7&2}{2&5&9&3&6&8&7&4&1}{7&3&6&9&1&4&2&8&5}{4&8&1&5&7&2&6&3&9}{3&6&8&1&9&5&4&2&7}{9&1&4&6&2&7&8&5&3}{5&7&2&8&4&3&1&9&6}
&\settablename{DS(9,134)}\namedsmallsudoku{1&4&7&2&5&9&3&6&8}{8&2&5&7&3&6&9&1&4}{6&9&3&4&8&1&5&7&2}{2&5&9&3&6&8&7&4&1}{7&3&6&9&1&4&2&8&5}{4&8&1&5&7&2&6&3&9}{3&7&4&1&9&5&8&2&6}{5&1&8&6&2&7&4&9&3}{9&6&2&8&4&3&1&5&7}
&\settablename{DS(9,135)}\namedsmallsudoku{1&4&7&2&5&9&3&6&8}{8&2&5&7&3&6&9&1&4}{6&9&3&4&8&1&5&7&2}{2&5&9&3&6&8&7&4&1}{7&3&6&9&1&4&2&8&5}{4&8&1&5&7&2&6&3&9}{3&7&4&8&2&5&1&9&6}{5&1&8&6&9&3&4&2&7}{9&6&2&1&4&7&8&5&3}
&\settablename{DS(9,136)}\namedsmallsudoku{1&4&7&2&5&9&3&6&8}{8&2&5&7&3&6&9&1&4}{6&9&3&4&8&1&5&7&2}{2&5&9&3&7&4&1&8&6}{7&3&6&5&1&8&4&2&9}{4&8&1&9&6&2&7&5&3}{3&7&4&8&2&5&6&9&1}{5&1&8&6&9&3&2&4&7}{9&6&2&1&4&7&8&3&5}
\\
\\
\settablename{DS(9,137)}\namedsmallsudoku{1&4&7&2&5&9&3&6&8}{8&2&5&7&3&6&9&1&4}{6&9&3&4&8&1&5&7&2}{2&5&9&3&7&4&6&8&1}{7&3&6&5&1&8&2&4&9}{4&8&1&9&6&2&7&3&5}{3&6&8&1&9&5&4&2&7}{9&1&4&6&2&7&8&5&3}{5&7&2&8&4&3&1&9&6}
&\settablename{DS(9,138)}\namedsmallsudoku{1&4&7&2&5&9&3&6&8}{8&2&5&7&3&6&9&1&4}{6&9&3&4&8&1&5&7&2}{2&5&9&3&7&4&6&8&1}{7&3&6&5&1&8&2&4&9}{4&8&1&9&6&2&7&3&5}{3&7&4&1&9&5&8&2&6}{5&1&8&6&2&7&4&9&3}{9&6&2&8&4&3&1&5&7}
&\settablename{DS(9,139)}\namedsmallsudoku{1&4&7&2&5&9&3&6&8}{8&2&5&7&3&6&9&1&4}{6&9&3&4&8&1&5&7&2}{2&5&9&6&1&7&4&8&3}{7&3&6&8&4&2&1&5&9}{4&8&1&3&9&5&7&2&6}{3&6&8&1&7&4&2&9&5}{9&1&4&5&2&8&6&3&7}{5&7&2&9&6&3&8&4&1}
&\settablename{DS(9,140)}\namedsmallsudoku{1&4&7&2&5&9&3&6&8}{8&2&5&7&3&6&9&1&4}{6&9&3&4&8&1&5&7&2}{2&5&9&6&1&7&4&8&3}{7&3&6&8&4&2&1&5&9}{4&8&1&3&9&5&7&2&6}{3&7&4&1&6&8&2&9&5}{5&1&8&9&2&4&6&3&7}{9&6&2&5&7&3&8&4&1}
\\
\\
\settablename{DS(9,141)}\namedsmallsudoku{1&4&7&2&5&9&3&6&8}{8&2&5&7&3&6&9&1&4}{6&9&3&4&8&1&5&7&2}{2&5&9&6&1&7&4&8&3}{7&3&6&8&4&2&1&5&9}{4&8&1&3&9&5&7&2&6}{3&7&4&5&2&8&6&9&1}{5&1&8&9&6&3&2&4&7}{9&6&2&1&7&4&8&3&5}
&\settablename{DS(9,142)}\namedsmallsudoku{1&4&7&2&5&9&3&6&8}{8&2&5&7&3&6&9&1&4}{6&9&3&4&8&1&5&7&2}{2&5&9&8&6&3&7&4&1}{7&3&6&1&9&4&2&8&5}{4&8&1&5&2&7&6&3&9}{3&7&4&9&1&5&8&2&6}{5&1&8&6&7&2&4&9&3}{9&6&2&3&4&8&1&5&7}
&\settablename{DS(9,143)}\namedsmallsudoku{1&4&7&2&5&9&3&6&8}{8&2&5&7&3&6&9&1&4}{6&9&3&4&8&1&5&7&2}{2&6&8&1&9&4&7&3&5}{9&3&4&5&2&7&6&8&1}{5&7&1&8&6&3&2&4&9}{3&5&9&6&1&8&4&2&7}{7&1&6&9&4&2&8&5&3}{4&8&2&3&7&5&1&9&6}
&\settablename{DS(9,144)}\namedsmallsudoku{1&4&7&2&5&9&3&6&8}{8&2&5&7&3&6&9&1&4}{6&9&3&4&8&1&5&7&2}{2&6&8&1&9&4&7&3&5}{9&3&4&5&2&7&6&8&1}{5&7&1&8&6&3&2&4&9}{3&8&6&9&1&5&4&2&7}{4&1&9&6&7&2&8&5&3}{7&5&2&3&4&8&1&9&6}
\end{tabular}
\end{center}

\pagebreak

\begin{center}
\begin{tabular}{P{28mm}P{28mm}P{28mm}P{28mm}}
\settablename{DS(9,145)}\namedsmallsudoku{1&4&7&2&5&9&3&6&8}{8&2&5&7&3&6&9&1&4}{6&9&3&4&8&1&5&7&2}{2&6&8&1&9&4&7&3&5}{9&3&4&5&2&7&6&8&1}{5&7&1&8&6&3&2&4&9}{3&8&6&9&4&2&1&5&7}{4&1&9&3&7&5&8&2&6}{7&5&2&6&1&8&4&9&3}
&\settablename{DS(9,146)}\namedsmallsudoku{1&4&7&2&5&9&3&6&8}{8&2&5&7&3&6&9&1&4}{6&9&3&4&8&1&5&7&2}{2&6&8&1&9&4&7&5&3}{9&3&4&5&2&7&1&8&6}{5&7&1&8&6&3&4&2&9}{3&5&9&6&1&8&2&4&7}{7&1&6&9&4&2&8&3&5}{4&8&2&3&7&5&6&9&1}
&\settablename{DS(9,147)}\namedsmallsudoku{1&4&7&2&5&9&3&6&8}{8&2&5&7&3&6&9&1&4}{6&9&3&4&8&1&5&7&2}{2&6&8&1&9&4&7&5&3}{9&3&4&5&2&7&1&8&6}{5&7&1&8&6&3&4&2&9}{3&5&9&6&7&2&8&4&1}{7&1&6&3&4&8&2&9&5}{4&8&2&9&1&5&6&3&7}
&\settablename{DS(9,148)}\namedsmallsudoku{1&4&7&2&5&9&3&6&8}{8&2&5&7&3&6&9&1&4}{6&9&3&4&8&1&5&7&2}{2&6&8&3&4&7&1&5&9}{9&3&4&8&1&5&7&2&6}{5&7&1&6&9&2&4&8&3}{3&5&9&1&6&8&2&4&7}{7&1&6&9&2&4&8&3&5}{4&8&2&5&7&3&6&9&1}
\\
\\
\settablename{DS(9,149)}\namedsmallsudoku{1&4&7&2&5&9&3&6&8}{8&2&5&7&3&6&9&1&4}{6&9&3&4&8&1&5&7&2}{2&6&8&3&4&7&1&5&9}{9&3&4&8&1&5&7&2&6}{5&7&1&6&9&2&4&8&3}{3&8&6&1&7&4&2&9&5}{4&1&9&5&2&8&6&3&7}{7&5&2&9&6&3&8&4&1}
&\settablename{DS(9,150)}\namedsmallsudoku{1&4&7&2&5&9&3&6&8}{8&2&5&7&3&6&9&1&4}{6&9&3&4&8&1&5&7&2}{2&6&8&3&4&7&1&9&5}{9&3&4&8&1&5&6&2&7}{5&7&1&6&9&2&8&4&3}{3&5&9&1&7&4&2&8&6}{7&1&6&5&2&8&4&3&9}{4&8&2&9&6&3&7&5&1}
&\settablename{DS(9,151)}\namedsmallsudoku{1&4&7&2&5&9&3&6&8}{8&2&5&7&3&6&9&1&4}{6&9&3&4&8&1&5&7&2}{2&6&8&3&4&7&1&9&5}{9&3&4&8&1&5&6&2&7}{5&7&1&6&9&2&8&4&3}{3&8&6&9&2&4&7&5&1}{4&1&9&5&7&3&2&8&6}{7&5&2&1&6&8&4&3&9}
&\settablename{DS(9,152)}\namedsmallsudoku{1&4&7&2&5&9&3&6&8}{8&2&5&7&3&6&9&1&4}{6&9&3&4&8&1&5&7&2}{2&6&8&3&9&5&7&4&1}{9&3&4&6&1&7&2&8&5}{5&7&1&8&4&2&6&3&9}{3&5&9&1&6&8&4&2&7}{7&1&6&9&2&4&8&5&3}{4&8&2&5&7&3&1&9&6}
\\
\\
\settablename{DS(9,153)}\namedsmallsudoku{1&4&7&2&5&9&3&6&8}{8&2&5&7&3&6&9&1&4}{6&9&3&4&8&1&5&7&2}{2&6&8&3&9&5&7&4&1}{9&3&4&6&1&7&2&8&5}{5&7&1&8&4&2&6&3&9}{3&8&6&9&2&4&1&5&7}{4&1&9&5&7&3&8&2&6}{7&5&2&1&6&8&4&9&3}
&\settablename{DS(9,154)}\namedsmallsudoku{1&4&7&2&5&9&3&6&8}{8&2&5&7&3&6&9&1&4}{6&9&3&4&8&1&5&7&2}{2&6&8&5&9&3&7&4&1}{9&3&4&1&6&7&2&8&5}{5&7&1&8&2&4&6&3&9}{3&5&9&6&1&8&4&2&7}{7&1&6&9&4&2&8&5&3}{4&8&2&3&7&5&1&9&6}
&\settablename{DS(9,155)}\namedsmallsudoku{1&4&7&2&5&9&3&6&8}{8&2&5&7&3&6&9&1&4}{6&9&3&4&8&1&5&7&2}{2&7&4&3&9&5&6&8&1}{5&3&8&6&1&7&2&4&9}{9&6&1&8&4&2&7&3&5}{3&5&9&1&6&8&4&2&7}{7&1&6&9&2&4&8&5&3}{4&8&2&5&7&3&1&9&6}
&\settablename{DS(9,156)}\namedsmallsudoku{1&4&7&2&5&9&3&8&6}{8&2&5&7&3&6&4&1&9}{6&9&3&4&8&1&7&5&2}{2&5&8&3&6&7&1&9&4}{9&3&6&8&1&4&5&2&7}{4&7&1&5&9&2&8&6&3}{3&6&9&1&4&8&2&7&5}{7&1&4&9&2&5&6&3&8}{5&8&2&6&7&3&9&4&1}
\\
\\
\settablename{DS(9,157)}\namedsmallsudoku{1&4&7&2&5&9&3&8&6}{8&2&5&7&3&6&4&1&9}{6&9&3&4&8&1&7&5&2}{2&5&8&3&6&7&1&9&4}{9&3&6&8&1&4&5&2&7}{4&7&1&5&9&2&8&6&3}{3&6&9&1&7&5&2&4&8}{7&1&4&6&2&8&9&3&5}{5&8&2&9&4&3&6&7&1}
&\settablename{DS(9,158)}\namedsmallsudoku{1&4&7&2&5&9&3&8&6}{8&2&5&7&3&6&4&1&9}{6&9&3&4&8&1&7&5&2}{2&5&8&3&6&7&9&4&1}{9&3&6&8&1&4&2&7&5}{4&7&1&5&9&2&6&3&8}{3&6&9&1&4&8&5&2&7}{7&1&4&9&2&5&8&6&3}{5&8&2&6&7&3&1&9&4}
&\settablename{DS(9,159)}\namedsmallsudoku{1&4&7&2&5&9&3&8&6}{8&2&5&7&3&6&4&1&9}{6&9&3&4&8&1&7&5&2}{2&5&8&3&6&7&9&4&1}{9&3&6&8&1&4&2&7&5}{4&7&1&5&9&2&6&3&8}{3&6&9&1&7&5&8&2&4}{7&1&4&6&2&8&5&9&3}{5&8&2&9&4&3&1&6&7}
&\settablename{DS(9,160)}\namedsmallsudoku{1&4&7&2&5&9&3&8&6}{8&2&5&7&3&6&4&1&9}{6&9&3&4&8&1&7&5&2}{2&5&8&3&9&4&1&6&7}{9&3&6&5&1&7&8&2&4}{4&7&1&8&6&2&5&9&3}{3&8&4&9&2&5&6&7&1}{5&1&9&6&7&3&2&4&8}{7&6&2&1&4&8&9&3&5}
\\
\\
\settablename{DS(9,161)}\namedsmallsudoku{1&4&7&2&5&9&3&8&6}{8&2&5&7&3&6&4&1&9}{6&9&3&4&8&1&7&5&2}{2&5&8&3&9&4&6&7&1}{9&3&6&5&1&7&2&4&8}{4&7&1&8&6&2&9&3&5}{3&6&9&1&7&5&8&2&4}{7&1&4&6&2&8&5&9&3}{5&8&2&9&4&3&1&6&7}
&\settablename{DS(9,162)}\namedsmallsudoku{1&4&7&2&5&9&3&8&6}{8&2&5&7&3&6&4&1&9}{6&9&3&4&8&1&7&5&2}{2&5&8&3&9&4&6&7&1}{9&3&6&5&1&7&2&4&8}{4&7&1&8&6&2&9&3&5}{3&8&4&9&2&5&1&6&7}{5&1&9&6&7&3&8&2&4}{7&6&2&1&4&8&5&9&3}
&\settablename{DS(9,163)}\namedsmallsudoku{1&4&7&2&5&9&3&8&6}{8&2&5&7&3&6&4&1&9}{6&9&3&4&8&1&7&5&2}{2&5&8&6&7&3&1&9&4}{9&3&6&1&4&8&5&2&7}{4&7&1&9&2&5&8&6&3}{3&6&9&8&1&4&2&7&5}{7&1&4&5&9&2&6&3&8}{5&8&2&3&6&7&9&4&1}
&\settablename{DS(9,164)}\namedsmallsudoku{1&4&7&2&5&9&3&8&6}{8&2&5&7&3&6&4&1&9}{6&9&3&4&8&1&7&5&2}{2&5&8&6&7&3&1&9&4}{9&3&6&1&4&8&5&2&7}{4&7&1&9&2&5&8&6&3}{3&8&4&5&9&2&6&7&1}{5&1&9&3&6&7&2&4&8}{7&6&2&8&1&4&9&3&5}
\\
\\
\settablename{DS(9,165)}\namedsmallsudoku{1&4&7&2&5&9&3&8&6}{8&2&5&7&3&6&4&1&9}{6&9&3&4&8&1&7&5&2}{2&5&8&9&4&3&6&7&1}{9&3&6&1&7&5&2&4&8}{4&7&1&6&2&8&9&3&5}{3&8&4&5&9&2&1&6&7}{5&1&9&3&6&7&8&2&4}{7&6&2&8&1&4&5&9&3}
&\settablename{DS(9,166)}\namedsmallsudoku{1&4&7&2&5&9&3&8&6}{8&2&5&7&3&6&4&1&9}{6&9&3&4&8&1&7&5&2}{2&5&9&1&6&7&8&4&3}{7&3&6&8&2&4&1&9&5}{4&8&1&5&9&3&6&2&7}{3&6&8&9&1&5&2&7&4}{9&1&4&6&7&2&5&3&8}{5&7&2&3&4&8&9&6&1}
&\settablename{DS(9,167)}\namedsmallsudoku{1&4&7&2&5&9&3&8&6}{8&2&5&7&3&6&4&1&9}{6&9&3&4&8&1&7&5&2}{2&5&9&1&6&7&8&4&3}{7&3&6&8&2&4&1&9&5}{4&8&1&5&9&3&6&2&7}{3&6&8&9&4&2&5&7&1}{9&1&4&3&7&5&2&6&8}{5&7&2&6&1&8&9&3&4}
&\settablename{DS(9,168)}\namedsmallsudoku{1&4&7&2&5&9&3&8&6}{8&2&5&7&3&6&4&1&9}{6&9&3&4&8&1&7&5&2}{2&5&9&3&4&7&8&6&1}{7&3&6&8&1&5&2&9&4}{4&8&1&6&9&2&5&3&7}{3&6&8&1&7&4&9&2&5}{9&1&4&5&2&8&6&7&3}{5&7&2&9&6&3&1&4&8}
\end{tabular}
\end{center}

\pagebreak

\begin{center}
\begin{tabular}{P{28mm}P{28mm}P{28mm}P{28mm}}
\settablename{DS(9,169)}\namedsmallsudoku{1&4&7&2&5&9&3&8&6}{8&2&5&7&3&6&4&1&9}{6&9&3&4&8&1&7&5&2}{2&5&9&3&4&7&8&6&1}{7&3&6&8&1&5&2&9&4}{4&8&1&6&9&2&5&3&7}{3&6&8&9&2&4&1&7&5}{9&1&4&5&7&3&6&2&8}{5&7&2&1&6&8&9&4&3}
&\settablename{DS(9,170)}\namedsmallsudoku{1&4&7&2&5&9&3&8&6}{8&2&5&7&3&6&4&1&9}{6&9&3&4&8&1&7&5&2}{2&5&9&3&6&8&1&7&4}{7&3&6&9&1&4&5&2&8}{4&8&1&5&7&2&9&6&3}{3&6&8&1&4&7&2&9&5}{9&1&4&8&2&5&6&3&7}{5&7&2&6&9&3&8&4&1}
&\settablename{DS(9,171)}\namedsmallsudoku{1&4&7&2&5&9&3&8&6}{8&2&5&7&3&6&4&1&9}{6&9&3&4&8&1&7&5&2}{2&5&9&3&6&8&1&7&4}{7&3&6&9&1&4&5&2&8}{4&8&1&5&7&2&9&6&3}{3&6&8&1&9&5&2&4&7}{9&1&4&6&2&7&8&3&5}{5&7&2&8&4&3&6&9&1}
&\settablename{DS(9,172)}\namedsmallsudoku{1&4&7&2&5&9&3&8&6}{8&2&5&7&3&6&4&1&9}{6&9&3&4&8&1&7&5&2}{2&5&9&3&7&4&1&6&8}{7&3&6&5&1&8&9&2&4}{4&8&1&9&6&2&5&7&3}{3&6&8&1&9&5&2&4&7}{9&1&4&6&2&7&8&3&5}{5&7&2&8&4&3&6&9&1}
\\
\\
\settablename{DS(9,173)}\namedsmallsudoku{1&4&7&2&5&9&3&8&6}{8&2&5&7&3&6&4&1&9}{6&9&3&4&8&1&7&5&2}{2&5&9&3&7&4&1&6&8}{7&3&6&5&1&8&9&2&4}{4&8&1&9&6&2&5&7&3}{3&7&4&8&2&5&6&9&1}{5&1&8&6&9&3&2&4&7}{9&6&2&1&4&7&8&3&5}
&\settablename{DS(9,174)}\namedsmallsudoku{1&4&7&2&5&9&3&8&6}{8&2&5&7&3&6&4&1&9}{6&9&3&4&8&1&7&5&2}{2&5&9&6&1&7&8&3&4}{7&3&6&8&4&2&5&9&1}{4&8&1&3&9&5&2&6&7}{3&6&8&9&2&4&1&7&5}{9&1&4&5&7&3&6&2&8}{5&7&2&1&6&8&9&4&3}
&\settablename{DS(9,175)}\namedsmallsudoku{1&4&7&2&5&9&3&8&6}{8&2&5&7&3&6&4&1&9}{6&9&3&4&8&1&7&5&2}{2&5&9&6&1&7&8&3&4}{7&3&6&8&4&2&5&9&1}{4&8&1&3&9&5&2&6&7}{3&7&4&1&6&8&9&2&5}{5&1&8&9&2&4&6&7&3}{9&6&2&5&7&3&1&4&8}
&\settablename{DS(9,176)}\namedsmallsudoku{1&4&7&2&5&9&3&8&6}{8&2&5&7&3&6&4&1&9}{6&9&3&4&8&1&7&5&2}{2&5&9&8&6&3&1&4&7}{7&3&6&1&9&4&8&2&5}{4&8&1&5&2&7&6&9&3}{3&6&8&9&1&5&2&7&4}{9&1&4&6&7&2&5&3&8}{5&7&2&3&4&8&9&6&1}
\\
\\
\settablename{DS(9,177)}\namedsmallsudoku{1&4&7&2&5&9&3&8&6}{8&2&5&7&3&6&4&1&9}{6&9&3&4&8&1&7&5&2}{2&5&9&8&6&3&1&4&7}{7&3&6&1&9&4&8&2&5}{4&8&1&5&2&7&6&9&3}{3&7&4&9&1&5&2&6&8}{5&1&8&6&7&2&9&3&4}{9&6&2&3&4&8&5&7&1}
&\settablename{DS(9,178)}\namedsmallsudoku{1&4&7&2&5&9&3&8&6}{8&2&5&7&3&6&4&1&9}{6&9&3&4&8&1&7&5&2}{2&6&8&1&9&4&5&7&3}{9&3&4&5&2&7&1&6&8}{5&7&1&8&6&3&9&2&4}{3&5&9&6&1&8&2&4&7}{7&1&6&9&4&2&8&3&5}{4&8&2&3&7&5&6&9&1}
&\settablename{DS(9,179)}\namedsmallsudoku{1&4&7&2&5&9&3&8&6}{8&2&5&7&3&6&4&1&9}{6&9&3&4&8&1&7&5&2}{2&6&8&1&9&4&5&7&3}{9&3&4&5&2&7&1&6&8}{5&7&1&8&6&3&9&2&4}{3&8&6&9&1&5&2&4&7}{4&1&9&6&7&2&8&3&5}{7&5&2&3&4&8&6&9&1}
&\settablename{DS(9,180)}\namedsmallsudoku{1&4&7&2&5&9&3&8&6}{8&2&5&7&3&6&4&1&9}{6&9&3&4&8&1&7&5&2}{2&6&8&3&4&7&1&9&5}{9&3&4&8&1&5&6&2&7}{5&7&1&6&9&2&8&4&3}{3&5&9&1&6&8&2&7&4}{7&1&6&9&2&4&5&3&8}{4&8&2&5&7&3&9&6&1}
\\
\\
\settablename{DS(9,181)}\namedsmallsudoku{1&4&7&2&5&9&3&8&6}{8&2&5&7&3&6&4&1&9}{6&9&3&4&8&1&7&5&2}{2&6&8&3&4&7&1&9&5}{9&3&4&8&1&5&6&2&7}{5&7&1&6&9&2&8&4&3}{3&8&6&9&2&4&5&7&1}{4&1&9&5&7&3&2&6&8}{7&5&2&1&6&8&9&3&4}
&\settablename{DS(9,182)}\namedsmallsudoku{1&4&7&2&5&9&3&8&6}{8&2&5&7&3&6&4&1&9}{6&9&3&4&8&1&7&5&2}{2&6&8&3&7&4&5&9&1}{9&3&4&5&1&8&2&6&7}{5&7&1&9&6&2&8&3&4}{3&5&9&6&2&7&1&4&8}{7&1&6&8&4&3&9&2&5}{4&8&2&1&9&5&6&7&3}
&\settablename{DS(9,183)}\namedsmallsudoku{1&4&7&2&5&9&3&8&6}{8&2&5&7&3&6&4&1&9}{6&9&3&4&8&1&7&5&2}{2&6&9&1&7&4&8&3&5}{7&3&4&5&2&8&6&9&1}{5&8&1&9&6&3&2&4&7}{3&7&6&8&1&5&9&2&4}{4&1&8&6&9&2&5&7&3}{9&5&2&3&4&7&1&6&8}
&\settablename{DS(9,184)}\namedsmallsudoku{1&4&7&2&5&9&3&8&6}{8&2&5&7&3&6&4&1&9}{6&9&3&4&8&1&7&5&2}{2&6&9&3&7&5&1&4&8}{7&3&4&6&1&8&9&2&5}{5&8&1&9&4&2&6&7&3}{3&7&6&8&2&4&5&9&1}{4&1&8&5&9&3&2&6&7}{9&5&2&1&6&7&8&3&4}
\\
\\
&\settablename{DS(9,185)}\namedsmallsudoku{1&4&7&2&5&9&3&8&6}{8&2&5&7&3&6&4&1&9}{6&9&3&4&8&1&7&5&2}{2&6&9&8&4&3&5&7&1}{7&3&4&1&9&5&2&6&8}{5&8&1&6&2&7&9&3&4}{3&7&6&9&1&4&8&2&5}{4&1&8&5&7&2&6&9&3}{9&5&2&3&6&8&1&4&7}
&\settablename{DS(9,186)}\namedsmallsudoku{1&4&7&2&5&9&3&8&6}{8&2&5&7&3&6&4&1&9}{6&9&3&4&8&1&7&5&2}{2&7&6&8&1&4&5&9&3}{4&3&8&5&9&2&1&6&7}{9&5&1&3&6&7&8&2&4}{3&6&9&1&7&5&2&4&8}{7&1&4&6&2&8&9&3&5}{5&8&2&9&4&3&6&7&1}
&
\end{tabular}
\end{center}

\bigskip\bigskip

\end{document}